\documentclass[11pt,reqno]{amsart}
\usepackage{graphicx,xcolor,cite,mathtools,bbold}
\usepackage{amssymb,amsmath,amsthm,mathrsfs,paralist,bm,esint,setspace}
\RequirePackage[colorlinks,citecolor=blue,urlcolor=blue]{hyperref}
\usepackage{cases}
\usepackage{color}
\usepackage{setspace}
\let\OLDthebibliography\thebibliography
\renewcommand\thebibliography[1]{
  \OLDthebibliography{#1}
  \setlength{\parskip}{3pt}
  \setlength{\itemsep}{0pt plus 0.3ex}
}

\newcommand{\<}{\langle}
\renewcommand{\>}{\rangle}

\newcommand{\N}{\mathbb{N}}
\newcommand{\Z}{\mathbb{Z}}

\newcommand{\R}{\mathbb{R}}

\newcommand{\T}{\mathbb{T}}

\newcommand{\lip}{\text{\rm Lip}}

\renewcommand{\P}{\mathrm{P}}
\newcommand{\E}{\mathrm{E}}

\newcommand{\cW}{\mathcal{W}}

\renewcommand{\d}{{\rm d}}

\newcommand{\e}{{\rm e}}
\renewcommand{\geq}{\geqslant}
\renewcommand{\leq}{\leqslant}
\renewcommand{\ge}{\geqslant}
\renewcommand{\le}{\leqslant}

\title[An invariance principle for SPDEs]{An Invariance Principle for some\\%
Reaction-Diffusion Equations with a Multiplicative Random Source} 
\author[D. Khoshnevisan]{Davar Khoshnevisan}
	\address{Department of Mathematics, University of Utah,
		Salt Lake City, Utah 84112-0090, USA}
	\email{davar@math.utah.edu}
\author[K. Kim]{Kunwoo Kim}
	\address{Department of  Mathematics, Pohang University of Science 
	and Technology (POSTECH), Pohang, Gyeongbuk, 37673, Korea} 
	\email{kunwoo@postech.ac.kr}
\author[C. Mueller]{Carl Mueller}
	\address{Department of Mathematics, University of Rochester,
	Rochester, New York 14627, USA}
	\email{carl.e.mueller@rochester.edu}
\thanks{Research supported in part by the United States National Science Foundation (DMS-2245242) and
	the National Research Foundation of Korea (2019R1A5A1028324 and RS-2023-00244382)}

\newtheorem{stat}{Statement}[section]
\newtheorem{proposition}[stat]{Proposition}

\newtheorem{corollary}[stat]{Corollary}
\newtheorem{theorem}[stat]{Theorem}
\newtheorem{lemma}[stat]{Lemma}

\newtheorem*{claim}{Claim}

\theoremstyle{definition} 

\newtheorem{assumption}[stat]{Assumption}

\newtheorem{remark}[stat]{Remark}

\numberwithin{equation}{section}

\begin{document}
\maketitle
\begin{abstract}
	We establish a notion of universality for the parabolic
	Anderson model via an invariance principle for a wide family
	of parabolic stochastic partial differential equations. We then use
	this invariance principle in order to provide an asymptotic theory
	for a wide class of non-linear SPDEs.  A novel ingredient
	of this invariance principle is the dissipativity of the underlying
	stochastic PDE.
\end{abstract}
\section{Introduction}
We consider reaction-diffusion equations of the  type
\begin{equation}\label{SHE}
\left[\begin{split}
	&\partial_t w = \partial^2_x w + f(w) +  g(w)\dot{W}\hskip.57in\text{on $(0\,,\infty)\times\T$},\\
	&\text{subject to }w(0)=w_0\hskip1.43in\text{on $\T$},
\end{split}\right.
\end{equation}
where $\T=(\R/2\Z)\cong [-1\,,1]$ denotes a torus\footnote{%
	Equivalently, we can think of $\T$ as the interval
	$[-1\,,1]$ and impose periodic boundary conditions on \eqref{SHE}.}
and $w=w(t\,,x)$ for $t\ge0$ and $x\in\T$.  
Fife \cite{Fife} contains a masterly account of
the general theory of reaction-diffusions equations. 
As a brief summary we mention only that
equations such as \eqref{SHE} arise prominently in the study of chemical reactions,
in which case the function $f$ encodes the nature of the underlying reaction.
Here, we are considering the case that the force has the form $g(w)\dot{W}$ 
for a nice function $g$ and a space-time white noise 
$\dot{W}=\{\dot{W}(t\,,x)\}_{t\ge0,x\in\T}$. In this way, \eqref{SHE} 
will be viewed as a stochastic PDE (SPDE)
in the sense of Walsh \cite{Walsh}.

Let us assume \emph{a priori} that \eqref{SHE} is (strongly) dissipative; that is,
\begin{equation}\label{Diss}\textstyle
	\lim_{t\to\infty} \sup_{x\in\T} |w(t\,,x)|=0
	\quad\text{almost surely}.
\end{equation}
Suppose also that $f$ and $g$ are continuously differentiable near the origin and both
vanish at the origin. Then a linear approximation of $f$ and $g$ at 
the origin suggests that, with probability one, 
\[
	f(w(t))\approx f'(0)w(t)
	\quad\text{and}\quad
	g(w(t))\approx g'(0)w(t)
	\quad\text{as $t\to\infty$}, 
\]
where, for a function $F(t\,,x)$, we write $F(t)$ to denote
$x\mapsto F(t\,,x)$.  Also, ``$A(t)\approx B(t)$'' loosely means
``$A(t)=B(t)+$ smaller-order terms.''
Thus, we can expect a rigorously stated version
of the following large-time approximation: With probability one,
\begin{equation}\label{SHE:approx}
	\partial_t w(t) \approx \partial^2_x w(t) + f'(0)w(t) + g'(0)w(t)
	\dot{W}(t)\qquad\forall t\gg1.
\end{equation}
Then the preceding, and some sort of ``asymptotic stability,'' together might suggest that
$w(t)\approx u(t)$ when $t\gg1$, and when $u=\{u(t\,,x)\}_{t\ge0,x\in\T}$ solves the linear stochastic PDE (or SPDE),
\begin{equation}\label{PAM}\left[\begin{split}
	&\partial_t u = \partial^2_x u + \mu u +  \sigma u\dot{W}\hskip.57in\text{on $(0\,,\infty)\times\T$},\\
	&\text{subject to }u(0)=w_0\hskip1.43in\text{on $\T$},
\end{split}\right.\end{equation}
where  here, and from now on,
\begin{equation}\label{mu:sigma}
	\mu = f'(0+)
	\quad\text{and}\quad
	\sigma=g'(0+).
\end{equation}

SPDEs of the type \eqref{PAM} are examples
of the \emph{parabolic Anderson model} (or PAM), with drift $\mu$ and
diffusion coefficient $\sigma$, so named by  Carmona and Molchanov \cite{CM94}
in the case that $(\partial^2_x\,,\T)$ is replaced by $(\Delta_d\,,\Z^d)$ where $\Delta_d$
denotes the graph Laplacian on $\Z^d$. 
And if the preceding dissipation argument were true, 
then it would prove a notion of universality of PAM by way of 
establishing an invariance principle that says that at large times the family \eqref{SHE}
indexed by the infinite-dimensional parameter 
space of all $(f,g)\in C^1_{\textit{loc}}(\T)\times C^1_{\textit{loc}}(\T)$ is
approximately  \eqref{PAM}
indexed by the two-dimensional parameter space of all $(\mu\,,\sigma)\in\R\times(\R\setminus\{0\})$. We pause to mention also that such a  property (if true) can be understood
as an invariance principle in the same vein as Donsker's theorem which states that
if $S_n=X_1+\cdots +X_n$ is a sum of $n$ i.i.d.\ random variables in $L^2(\Omega)$, then 
(on a suitable probability space) $S_n\approx mn +s B_n$ when $n\gg1$,
where $m=\E X_1$ and $s^2=\text{Var}X_1$, and $B=\{B_t\}_{t\ge0}$ is 
a standard Brownian motion on the line. Thus we see that the law of $S_n$ is immaterial
when $n\gg1$ except for two associated parameters $m$ and $s$.

Unfortunately, the outlined dissipation argument for weak universality is flawed, for example
because:
\begin{compactenum}
	\item In contrast with classical partial differential equations, 
		SPDEs are not local equations because they lack the requisite 
		{\it a priori} smoothness. This means that \eqref{SHE:approx} does not have
		a natural rigorous meaning, and so the argument that led to
		\eqref{SHE:approx} fails promptly. 
	\item It is unlikely that there exists a suitable large-time stability result that
		would imply that even a generous interpretation of \eqref{SHE:approx}
		would imply that  $w(t) \approx u(t)$ as $t\to\infty$. And such a result 
		certainly does not currently exist.
\end{compactenum}

Despite the flaws of the above dissipation argument,
the goal of this paper is to prove that the end result, namely the 
approximation of \eqref{SHE} by \eqref{PAM} at large
times, is nevertheless true under a broad set of conditions on $f$ and $g$. Under such conditions,
we will prove a statement that is only slightly weaker than, but strong enough to have
all of the consequences that we would like from, the following: On a suitably
constructed probability space, it is possible to construct $u$ and $w$ such that
\[
	\lim_{t\to\infty} u(t)/w(t)=1,
	\text{\ \  equivalently,\ \  }\log u(t)-\log w(t)\to 0\text{ as }t\to\infty;
\]
see Theorem \ref{th:main} for the precise statement. We will see also that the conditions on
$f$ and $g$  have ``physical'' meanings and, among other things, imply
the dissipativity \eqref{Diss} of the solution to \eqref{SHE}. We will say 
some things about this topic in \S\ref{sec:cor} below. Our method of proof requires us
to first develop a number of preliminary results and is therefore a little too involved
to summarize in the Introduction. The introductory portion of
\S\ref{sec:Pf} contains a summary outline of our methods, easier to discuss
once some of the preliminary results are established. The summary is then
followed by the details of the proof itself.

The linear SPDE \eqref{PAM} has now a well-developed asymptotic theory; see
Brunet, Gu, and Komorowski \cite{GK2} and
Gu and Komorowski \cite{GK,GK2022}. An up-to-date summary of these results,
and more, can be found in a recent survey article by Gu and Komorowski \cite{GK2024}.
When combined with those works, the results of this paper
yield an associated large-$t$ theory for a wide class of
dissipative solutions to the nonlinear equation \eqref{SHE};
see Corollaries \ref{cor:LLN} and \ref{cor:CLT} below. It would be interesting to know 
whether one can develop a parallel theory for the same equations but where the spatial variable
$x$ is in $\R$ and not $\T$. It is possible that one can use large-scale local dissipation
results such as those in our recent work \cite{KKM2024} in place of the dissipation phenomenon 
here in order to relate \eqref{SHE} on $\R_+\times\R$ to its linearized form \eqref{PAM} on
$\R_+\times\R$. If this can be done then one can hope
that the recent discoveries of Dauvergne, Ortmann, and Vir\'ag \cite{DauvergneOrtmannVirag} and
Quastel and Sarkar \cite{QuastelSarkar} might yield the Brownian landscape as the limit of a large
number of locally dissipative solutions to \eqref{SHE}.

An outline of the paper follows: In \S\ref{sec:main} we identify the key technical hypotheses
of this paper and use them to state our principal result. A few quick corollaries (law of large
numbers and a central limit theorem) are mentioned briefly in \S\ref{sec:cor}. 
Several essential technical facts about the SPDE \ref{SHE} are collected in \S\ref{sec:facts} below. 
The main result of this paper hinges on a coupling method that is described subsequently in 
\S\ref{sec:coupling}, followed by the proof of our main theorem (Theorem \ref{th:main}) which
appears in \S\ref{sec:Pf}.%

We conclude the Introduction by setting forth some notation that we plan to use 
without further mention.%

Throughout, $\log$ denotes the natural logarithm, and 
$\log_+ x = \log(x\vee\e)$ for all $x\ge0$.
For all $x,y\in\T$ and $t>0$, we define the heat kernel as
\begin{equation}\label{p:G}
	p_t(x\,,y) =  \sum_{n=-\infty}^\infty G_t(x-y+2n),
	\text{ for }\\
	G_t(a) =  (4\pi t^2)^{-1/2} \e^{-a^2/(4t)},
\end{equation}
valid for every $a\in\R$. Note that $p_t(x\,,y)=p_t(y\,,x)$
depends only on $x-y$, where we are writing abelian-group operations
on $\T$ using additive notation as is wont. With this in mind, we frequently write
\[
	p_t(y-x) = p_t(x-y) = p_t(x\,,y)\qquad\forall t>0,\ x,y\in\T,
\]
as is customary in the literature of L\'evy processes/random walks. In this way we
may identify the semigroup as a convolution semigroup:
$(p_t*\phi)(x) = \int_\T p_t(x\,,y)\phi(y)\,\d y,$
where ``$*$'' denotes convolution on $L^2(\T)$.
With this notation in mind, 
let us also recall that \eqref{SHE} is shorthand for the mild form,
\begin{equation}\label{eq:mild}\begin{split}
	w(t\,,x) &\textstyle = (p_t*w_0)(x)
          + \int_{(0,t)\times\T} p_{t-s}(x\,,y)
	 	f(w(s\,,y))\,\d s\,\d y\\
	 &\textstyle\hskip1.2in+ \int_{(0,t)\times\T} p_{t-s}(x\,,y)
		 g(w(s\,,y))\,\dot{W}(\d s\,\d y),
\end{split}\end{equation}

Whenever $\mathscr{G}\subseteq\mathscr{F}$ is a sub sigma-algebra of the underlying probability
space $(\Omega\,,\mathscr{F},\P)$ we might write
\[
	\E_{\mathscr{G}} (X) = \E(X\mid\mathscr{G})
	\qquad\forall X\in L^1(\Omega) = L^1(\Omega\,,\mathscr{F},\P),
\]
in order to simplify the exposition by shortening some long formulas. Similarly, 
$\P_{\mathscr{G}}$ might on occasion denote the conditional probability $\P(\,\cdots\mid\mathscr{G})$.

Throughout, we define $\{\mathscr{F}(t)\}_{t\ge0}$ to be the
filtration associated to the space-time
white noise $\dot{W}$; that is,  define $\mathscr{F}(t)$ 
to be the $\sigma$-algebra generated by all Wiener integrals of the form
$\int_{(0,s)\times\T}\varphi(y)\,W(\d s\,\d y)$ as $(s\,,\varphi)$ 
ranges over $[0\,,t]\times L^2(\T)$.%

\section{The Main Result}\label{sec:main}

In this section we set up the hypotheses under which we can prove an invariance principle,
and then follow them with a precise statement of that theorem. Recall $\mu,\sigma\in\R$
from \eqref{mu:sigma}, and let
\begin{equation}\label{R:star}\textstyle
	\mathscr{E}(a) = \sup_{z\in(0,a)}\left| z^{-1}f(z) - \mu\right| + 
	\sup_{z\in(0,a)}\left| z^{-1}g(z) - \sigma\right|
	\quad\forall a>0.
\end{equation}
The function $\mathscr{E}$ can be locally finite only when $f(0)=g(0)=0$,
in which case $\mathscr{E}$ may serve as a modulus of continuity 
for both $f'$ and $g'$ to the right of zero. The following formulates many of
the basic assumptions under which our results hold and, in particular, render
$\mathscr{E}$ a modulus of continuity in the sense mentioned above.

\begin{assumption}\label{ass:par}	
	The initial profile $w_0:\T\to\R_+$ can
	be random but must be independent of $\dot{W}$,
	and $\sup_\T w_0 \in L^k(\Omega)$ for every $k\ge1$. We also assume
	that $\int_\T w_0(x)\,\d x>0$ a.s. The functions
	$f,g:\R\to\R$ are both nonrandom, $f(0)=g(0)=0$  and:
	\begin{compactenum}
		\item $f$ and $g$ are Lipschitz continuous on $\R$, both are
			differentiable from the right at the origin; 
		\item ${\rm L}_g = \inf_{z>0}|g(z)/z|>0$;
		\item Recall \eqref{R:star} and assume that there exists a number
			$\chi>2$ such that
			\[
				\mathscr{E}(z)  \lesssim |\log z|^{-\chi}
				\text{\ \ uniformly for all $z\in(0\,,1)$}.
			\]
	\end{compactenum}
\end{assumption}

Next we say a few things about Assumption \ref{ass:par}.

\begin{remark}
	Part (1) of Assumption \ref{ass:par} and the condition that 
	$\sup_\T w_0 \in L^k(\Omega)$ for every $k\ge1$ are enough to give 
	existence and uniqueness for \eqref{SHE};  see Walsh \cite[Chapter 3]{Walsh}.
	Part (2) of Assumption \ref{ass:par} is an intermittency
	type condition; see Foondun and Khoshnevisan \cite{FoondunKhoshnevisan2009}. Finally,
	Part (3) of Assumption \ref{ass:par} holds whenever $f'$ and $g'$
	are locally H\"older continuous in a neighborhood to the right of the origin.
\end{remark}

We can now state our main result.

\begin{theorem}\label{th:main}
	Suppose that Assumption \ref{ass:par} holds, and
	recall the definition of the constant ${\rm L}_g$ from part (2) of 
        that assumption. Suppose that
	\begin{equation}\label{cond:f}
		\sup_{z>0} |f(z)/z|<  {\rm L}_g^2/64.
	\end{equation}
	Then, for every $\varepsilon\in(0\,,1)$
	there exist nonrandom numbers $S=S(\varepsilon)>0$ and $\beta
	=\beta(\varepsilon) \in (0\,, 1)$, 
	and a coupling $(u\,,w)$ -- that depends on the choice of $\varepsilon$ --
	where the space-time random field $w = \{w(t\,,x)\}_{t\ge0,x\in\T}$ solves \eqref{SHE} and 
	the space-time random field $u=\{u(t\,,x)\}_{t\ge S,x\in\T}$ is a solution to \eqref{PAM}
	shifted by $S$ time units and started at $w(S)$; that is,
	\begin{equation}\label{eq:main}\left[\begin{split}
		&\partial_t u = \partial^2_x u + \mu u  + \sigma u\dot{W}_0
			\hskip.45in\text{on $(S\,,\infty)\times\T$},\\
		&\text{subject to }u(0)=w(S)\hskip.9in\text{on $\T$},
	\end{split}\right.\end{equation}
	for a suitably constructed
	space-time white noise $\dot{W}_0$, such that
	\[\textstyle
		\P\left\{ \| \log w(t) - \log u(t) \|_{C(\T)} 
		=O( t^{-\beta})\text{\ as $t\to\infty$} \right\} \ge 
		1 - \varepsilon. 
	\]
\end{theorem}

It is possible to use Theorem \ref{th:main} and techniques from stochastic analysis
in order to improve Theorem \ref{th:main} itself. We will give an example below that is
particularly well motivated by the theory of reaction-diffusion equations and spatiotemporal
intermittency. Before we delve into that discussion, let us stop to make two brief remarks 
about Theorem \ref{th:main}.

\begin{remark}[KPZ asymptotics]
	 We can formally apply the Hopf-Cole transformation to the construction of 
	the parabolic Anderson model in Theorem \ref{th:main}:
	$h(t\,,x) = \log u(t\,,x)$ for all $t \ge S$ and $x\in\T$.
	It follows informally that $h$ ought to solve
	the KPZ equation (see Hairer \cite{Hairer2013} for a rigorous statement),
	$\partial_t h = \partial^2_x h + (\partial_x h)^2 + \mu + \sigma \dot{W}_0$
	on $(S\,,\infty)\times\T$, subject to $h(S)= \log w(S)$.	
	Consequently, we see that in the setting of Theorem \ref{th:main},
	$\lim_{t\to\infty}\,\| \log w(t) - h(t) \|_{C(\T)}=0$ with probability 
	at least $1-\varepsilon$. Since $\varepsilon$ is arbitrary, it follows that many
	large-time properties of $h$ that do not depend on its initial state are inherited
	almost surely by $\log w$.
\end{remark}
\begin{remark}[Level of the noise]\label{rem:noise}
	We may apply Theorem \ref{th:main} with $g$ replaced everywhere 
	by $\lambda g$ where $\lambda\in\R\setminus\{0\}$ and $g$ 
	is viewed as a fixed function that satisfies
	Assumption \ref{ass:par}. In this way we can think of $\lambda$
	as the ``level of the noise'' in the reaction-diffusion equation \eqref{SHE}.
	Condition \eqref{cond:f} can now be viewed as a quantitative way to say
	that $\lambda$ is sufficiently large. That is, Theorem \ref{th:main},
	and in fact the entire paper, is concerned solely with the ``high noise
	regime,'' which is known to loosely coincide with 
	some sort of ``dissipative regime''; see Liggett \cite[Chapter 12]{Liggett}
	for example.
\end{remark}

Theorem \ref{th:main} can sometimes hold when $f$ is not Lipschitz continuous.
Standard examples can include $f(x)= a x - bx^2$ (Fischer-KPP) or $f(x)=ax-bx^3$
(Allen-Cahn) where $a,b\in\R$ are fixed parameters. 
Since in such examples $f$ is not globally Lipschitz, let us spend a few lines explaining how they too
can fit within the scope of Theorem \ref{th:main}.
More generally, let us consider a function $f$ of the form
$f(x)=ax-V(x)$ where $V:\R\to\R_+$ is locally Lipschitz (and non negative). 
There is also an associated comparison
theorem which shows that $\P\{0\le w\le \bar{w}\}=1$ where $\bar{w}$ denotes the
solution to \eqref{SHE} when $f(x)$ is replaced everywhere by $ax$, equivalently $V$
is replaced by zero. Since $\bar{w}$ is dissipative when $|a|\le{\rm L}_g^2/64$
(see \eqref{cond:f}), it follows that $w$ is also dissipative. Therefore, we can retrofit the proof of Theorem 
\ref{th:main}, making only minor changes, in order to deduce that Theorem \ref{th:main}
is valid, when $f(x)=ax-V(x)$ for a nonnegative locally Lipschitz continuous 
$V$ that is differentiable to the right of zero, with $\mu=a-V'(0)$. In light of 
Remark \ref{rem:noise}, it follows that such SPDEs are asymptotically linearizable when
the level of the noise is sufficiently high. This refines a result of Khoshnevisan, Kim, Mueller,
and Shiu \cite{KKM2023a} which proved that 
\eqref{SHE} has a unique invariant measure (no phase transition) when the noise level is high. 
By contrast one can see in the
same reference \cite{KKM2023a} that, when the level of the noise is low, there are two distinct extremal
ergodic invariant measures (phase transition). Such results
were anticipated earlier in the work of Zimmermann, Toral, Piro, and San Miguel \cite{ZimmermannEtAl}
and proved first in \cite{KKM2023a}.
The existence of a sharp cut-off for the uniqueness of invariant measure remains open, though
the results of an elegant recent paper by Blessing and Rosati \cite{BlessingRosati}
suggests that a transition cut-off might indeed exist.

\section{Some corollaries}\label{sec:cor}
In this section we pause our discussion to record two noteworthy consequences
of Theorem \ref{th:main} and the recently developed asymptotic theory of
Gu and Komorowski \cite{GK,GK2024} and Brunet, Gu, and Komorowski \cite{GK2}
for the linear equation \eqref{PAM}. We also mention one among many
possible extensions of Theorem \ref{th:main}.

Choose and fix some $\varepsilon\in(0\,,1)$,
and define $\mathscr{U}(t\,,x) = u(t/2\,,x)$
for all $t\ge 2S$ and $x\in\T$,
where $(u\,,S)$ is given by Theorem \ref{th:main}. Then,
the scaling properties of white noise imply that $\mathscr{U}$ solves
the following variant of \eqref{PAM} for a suitably constructed space-time white noise
$\dot{\mathscr{W}}$ that is independent of $w(S)=\{w(S\,,x)\}_{x\in\T}$: 
\[\textstyle
	\partial_t\mathscr{U} =\frac12\partial^2_x\mathscr{U} 
	+ \frac12\mu\mathscr{U} + \frac{1}{\sqrt{2}}\sigma\mathscr{U}
	\dot{\mathscr{W}}
\]
on $(2S\,,\infty)\times\T$, subject to $\mathscr{U}(2S)=w(S).$
Since $\E(\|w(S)\|_{C(\T)}^2)<\infty$, we can apply
the theory of Gu and Komorowski \cite{GK}, and the subsequent computations of
Brunet, Gu, and Komorowski \cite{GK2}, with $L=2$ and $\beta=\sigma$ in 
order to see that, 
for every $x\in\T$,
\[\textstyle
	\lim_{t\to\infty}
	t^{-1}\log \mathscr{U}(t\,,x) = \frac12\mu - \gamma_2(\sigma), \quad
	\text{\ where } \gamma_2(\sigma)=
	\frac{\sigma^2}{8} \left( 1 + \frac{\sigma^2}{12}\right),
\]
and where the convergence holds in probability. In \cite{KKM2023:Oscillations} we showed
that the convergence in fact holds almost surely and uniformly in $x\in\T$. Thus, we see that,
in  general, $t^{-1}\log u(t\,,x)
 - \mu + 2\gamma_2(\sigma) \to0$ as $t\to\infty$ 
 uniformly in $x\in\T$ a.s.
Since the preceding limit theorem does not depend on $(\varepsilon\,,S)$, 
we may deduce the following from Theorem \ref{th:main}:
\begin{corollary}\label{cor:LLN}
	Under the conditions of Theorem \ref{th:main},
	\[
		\adjustlimits\lim_{t\to\infty} \sup_{x\in\T}\left| t^{-1}\log w(t\,,x)
		 - \mu + 2\gamma_2(\sigma)\right|=0
		\text{\quad a.s.}
	\]
\end{corollary}
Similarly, the central limit theorem of Gu and Komorowski \cite{GK} in the functional form that
we described in \cite{KKM2023:Oscillations} yields the following:
\begin{corollary}\label{cor:CLT}
	Under the conditions of Theorem \ref{th:main},
	\[
		\frac{\log w(t) - (\mu-2\gamma_2(\sigma))t\mathbb{1}}{\sqrt t}
		\Rightarrow\eta \mathbb{1}\quad\text{as $t\to\infty$},
	\]
	where $\mathbb{1}(x)=1$ for all $x\in\T$, $\eta$ has a non-degenerate,
	centered normal distribution, and $\text{\rm ``$\Rightarrow$''}$ denotes
	weak convergence in $C(\T)$.
\end{corollary}

\section{A Few Relevant Facts}\label{sec:facts}
In this section we collect some technical facts about the SPDE \eqref{SHE}.

\subsection{Moment Bounds}

\begin{lemma}\label{lem:Young}
	For $\varphi:\R_+\times\T\to\R$, formally define
	\[
		J_\varphi (t\,, x) =\int_{(0,t)\times\T}
		p_{t-s} (x-y) \varphi(s\,, y)\, \d y\, \d s
		\quad\forall t>0,\, x\in\T.
	\]
	If $\varphi\in L^k([0\,,T]\times\T)$ for some $k\in[2\,,\infty)$
	and $T>0$, then
	there is a version of $J_\varphi$  that is in $C([0\,,T]\times\T)$ and satisfies
	\[
		\| J_\varphi\|_{C([0,T]\times\T)} 
		\le 3\|\varphi\|_{L^k([0,T]\times\T)} \left( T^{(2k-3)/(2k)}\vee T\right).
	\] 
\end{lemma}

\begin{proof}
	By Young's convolution inequality for the circle group \cite[p.\ 333]{Young12},
	\begin{align*}
		\| J_\varphi\|_{C([0,T]\times\T)} &\le  
			\|\varphi\|_{L^k([0,T]\times\T)}
			\|p\|_{L^{k/(k-1)}([0,T]\times\T)}\\
		&\textstyle\le \|\varphi\|_{L^k([0,T]\times\T)}
                	\left(\int_0^T\|p_t\|_{C(\T)}^{1/(k-1)}\,\d t\right)^{(k-1)/k},
	\end{align*}
	because $|p_t(x)|^{k/(k-1)}\le \| p_t \|_{C(\T)}^{1/(k-1)} p_t(x)$ for all $t>0$ and $x\in\T$,
	and since $\int_\T p_t(x)\,\d x=1$. Thanks to Lemma B.1 of \cite{KKMS},
        $\|p_t\|_{C(\T)}\le 2( t^{-1/2}\vee1)$ for all $t>0$.
        In this way we find that if $T\in(0\,,1]$ and $k\ge2$, then 
        \begin{align*}
        	\|J_\varphi\|_{C([0,T]\times\T)}^{k/(k-1)} &\le\|\varphi\|_{L^k([0,T]\times\T)}^{k/(k-1)}
			\times 2^{1/(k-1)}\times \int_0^T\frac{\d t}{t^{1/\{2(k-1)\}}}\\
		&\le 2^{k/(k-1)}\|\varphi\|_{L^k([0,T]\times\T)}^{k/(k-1)}
			T^{(2k-3)/(2k-2)},
        \end{align*}
        since $(2k-2)/(2k-3)\le2$.
        And similarly if $T>1$ and $k\ge 2$, then 
        \begin{align*}
        	\|J_\varphi\|_{C([0,T]\times\T)}^{k/(k-1)} &\le
			2^{1/(k-1)}\|\varphi\|_{L^k([0,T]\times\T)}^{k/(k-1)}
			\left(\frac{2k-2}{2k-3} + T\right)\\
		&\le 3\times 2^{1/(k-1)}\|\varphi\|_{L^k([0,T]\times\T)}^{k/(k-1)}T,
	\end{align*} 
	since $(2k-2)/(2k-3)\le2 \le 2T$.
\end{proof}

We may formally define $I_X$ to be the stochastic convolution,
\begin{equation}\label{IX}
	I_X(t\,,x) = \int_{(0,t)\times\T} p_{t-s}(x\,,y) X(s\,,y)\,\dot{W}(\d s\,\d y),
\end{equation}
whenever $X$ is a predictable space-time random field in the sense of Walsh \cite{Walsh}.
In order for $I_X$ to be a properly defined random field, $X$ has to have additional
distributional properties. Those, and a few additional properties, are summarized in the following result.

\begin{lemma}\label{lem:Salins} 
	Suppose that $X=\{X(t\,, x)\}_{t>0,x\in\T}$ is a 
	predictable random field and formally define the stochastic convolution $I_X$ via
	\eqref{IX}.
	If\\ $\E(\|X\|_{L^k([0,T_*]\times\T)}^k)<\infty$ for some $T_*>0$
	and $k\in(6\,,\infty)$, then $I_X$ is a well-defined Walsh integral, and 
	has a continuous version that satisfies the following for all $T\in[0\,,T_*]$
	and all stopping times $\tau$:
	\begin{equation*}\textstyle
		\E \left( \| I_X \|_{C([0,T\wedge\tau]\times\T)}^k\right)
		\lesssim T^{(k-6)/4} \E\left( \int_0^T\d s\int_\T\d y\
		|X(s\wedge\tau\,,y)|^k\right),
	\end{equation*}
	where the implied constant depends only on $(T_*,k)$.
\end{lemma} 
This result is essentially due to Salins except that two minor adjustments need to be made:
\begin{compactenum}
	\item Salins \cite{Salins} states Lemma \ref{lem:Salins} in the case that the heat kernel on
		$\T\cong[-1\,,1]$ (see \eqref{p:G}) is replaced by $\pi\T\cong[-\pi\,,\pi]$. 
		One can scale the torus from $\pi\T$ to $\T$ without incurring changes
		(except the implied constant has to be adjusted). 
	\item In its original version \cite[Theorem 1.2]{Salins}, Lemma \ref{lem:Salins} is proved
		in the case that $\tau$ is a deterministic constant, say $\tau=1$, in order to estimate
		the norm of the supremum of the random field $(t\,,x)\mapsto I_X(t\,,x)$. 
		The factorization method employed in \cite{Salins} works just as well, and in essentially
		exactly the same way, if we apply it to the random field $(t\,,x)\mapsto I_X(t\wedge\tau\,,x)$.
\end{compactenum}

\subsection{Dissipation}

Recall Assumption \ref{ass:par}. In this section our objective is to
quantify the assertion that Condition \eqref{cond:f} implies that
\eqref{SHE} is dissipative in the sense that $\|w(t)\|_{C(\T)}\to0$
as $t\to\infty$. In fact, dissipation occurs rapidly with high probability, as the following shows.

\begin{lemma}\label{lem:dissipation}
	If Assumption \ref{ass:par} is in force, and if $f$ additionally satisfies \eqref{cond:f}, then
	there exists a constant $C=C(f)>1$ such that for every $m>0$ there exists $C_m>0$
	\[\textstyle
		\P\left\{   \|w(t)\|_{C(\T)} < \exp\left( - t\left[\frac{{\rm L}_g^2}{64}
		- \sup_{z>0}\frac{f(z)}{z} \right] \right) \ \  \forall t\ge T\right\}
            	\ge 1 - C_m\e^{-T/C},
        \]
        uniformly for all $T>0$ and all initial data $w_0$ that satisfy $\|w_0\|_{C(\T)}\le m$.
\end{lemma} 

\begin{proof}
	In order to simplify the exposition define
	$Q = \sup_{z>0}(f(z)/z)$. Condition \eqref{cond:f} ensures that
	$-\infty < Q < {\rm L}_g^2/64$. 
	
	Thanks to a comparison theorem (see for example \cite{Mueller1}, Theorem 3.1, or \cite{Shiga1994}), 
	$w\le \bar{w}$ a.s., where $\bar{w}$ satisfies
	the following SPDE:
	\begin{align*}
		&\partial_t \bar{w} = \partial^2_x \bar{w} +  Q \bar{w} + 
			g(\bar{w})\dot{W}\hskip.57in\text{on $(0\,,\infty)\times\T$},\\
		&\text{subject to }\bar{w}(0)=w_0\hskip0.87in\text{on $\T$}.
	\end{align*}
       
	It is a standard, easy-to-check fact that 
        $\bar{w}(t\,,x) = \exp(Qt)\cW(t\,,x)$
	where $\cW$ solves \eqref{SHE} with $f\equiv 0$ and  $g$ replaced by 
	$\bar g(t, y)=\exp(-Qt) g( \exp(Qt)y)$, i.e., 
	\begin{align*}
		&\partial_t \cW = \partial^2_x \cW + 
			\bar{g}(\cW)\dot{W}\hskip.57in\text{on $(0\,,\infty)\times\T$},\\
		&\text{subject to }\cW(0)=w_0\hskip0.56in\text{on $\T$}.
	\end{align*}
	Therefore, it remains to prove that, in the same setting of Lemma \ref{lem:dissipation},
	\[
		\P\left\{   \| \cW \|_{C(\T)} < \e^{-{\rm L}_g^2t/64}, \ \  \forall t\ge T\right\}
            	\ge 1 - C_m\e^{-T/C}.
        \]
	We can see from two displays above equation (7.1) of Khoshnevisan, Kim, Mueller, and Shiu \cite{KKMS}, 
	applied with their $(\varepsilon\,,\lambda\,,\rho)$ set to $(\frac12\,,1\,,1)$, 
	that  there exists a universal constant $c>0$ --- dubbed as $\mu$ in \cite{KKMS}
	instead of the present number $c$ --- such that
	the following holds uniformly for all $N\in\N$:
	\begin{align*}
		&\P\left\{  \sup_{s\in [\exp(Nc),\exp([N+1]c)] } \sup_{x\in\T} 
            		\left| \e^{ {\rm L}_g^2s/64} \cW(s\,,x) \right| \ge 1 ~;\, 
			\bm{A}(\e^{Nc-1};1)\right\}\\
		&\hskip2.7in\lesssim \left[1+\e^{Nc}\right]^{3/2}\exp\left( 
			-{\rm L}_g^2\e^{Nc}/64\right),
	\end{align*}
	where $\P( \bm{A}(t\,;1)) \ge 1 - \exp(-{\rm L}_g^2 t/64)$ for all $t>0;$
	see \cite[eq.\ (6.4)]{KKMS} where \text{the implied constant depends only 
	on $m$ and }one can find an explicit construction of 
	the event $\bm{A}(t\,;1)$ as well.  To be completely clear, the results of
	\cite{KKMS} are written in the case that the diffusion coefficient is time-independent.
	Here, $\bar{g}$ depends on time but the method (and results) of \cite{KKMS}
	work verbatim in this case as well since the Lipschitz constant of $\bar{g}(t)$,
	viewed as a spatial function for each $t>0$, is bounded uniformly in $t$.
	In fact, $\lip(\bar{g}(t))\le\lip(g)$ for all $t>0$.
	In this way we find that
	\begin{align*}
		&\textstyle\P\left\{  \sup_{s\in [\exp(Nc),\exp([N+1]c)] } \sup_{x\in\T} 
            		\left| \e^{ {\rm L}_g^2s/64}\cW(s\,,x) \right| \ge 1 \right\}\\
		&\textstyle\le \P\left\{  \sup_{s\in [\exp(Nc),\exp([N+1]c)] } \sup_{x\in\T} 
            		\left|\e^{ {\rm L}_g^2s/64}\cW(s\,,x) \right| \ge 1 ~;\, 
			\bm{A}(\e^{Nc-1};1)\right\} \\
		&\hskip1in + \P\left( \bm{A}(\e^{Nc-1};1)^\mathsf{c}\right) 
			\lesssim \exp\left( \tfrac32 Nc  - \tfrac{1}{64} {\rm L}_g^2\e^{Nc-1}\right),
	\end{align*}
	uniformly for all $N\in\N$. In particular, the following is valid uniformly
	for all $n\in\N$ and for all initial data that satisfy
	$\|w_0\|_{C(\T)}\le m$:
	\begin{align*}
		\textstyle\P\left\{  \sup_{s\ge \exp(nc)} \sup_{x\in\T} 
            		\left| \e^{ {\rm L}_g^2s/64}\cW(s\,,x) \right| \ge 1 \right\}
			&\textstyle\lesssim \sum_{N=n}^\infty
			\exp\left( \tfrac32 Nc - \tfrac{1}{64} {\rm L}_g^2\right)\\
		&\lesssim\exp\left(-  {\rm L}_g^2\e^{nc}/(65\e)\right).
	\end{align*}
	If $T\in[\exp([n-1]c)\,,\exp(nc)]$ for some $n\in\N$, then
	\[
		\P\left\{ \adjustlimits \sup_{s\ge T} \sup_{x\in\T} 
            	\left| \frac{\cW(s\,,x)}{\e^{- {\rm L}_g^2s/64}} \right| \ge 1 \right\}
		\le 
		\P\left\{ \adjustlimits \sup_{s\ge \exp(nc)} \sup_{x\in\T} 
            	\left| \frac{\cW(s\,,x)}{\e^{- {\rm L}_g^2s/64}} \right| \ge 1 \right\},
	\]
	and  $\exp(- {\rm L}_g^2\e^{(n-1)c}/(65\e))
	\le \exp(-  {\rm L}_g^2T/(65\e^{1+c})).$
	These observations together prove Lemma \ref{lem:dissipation}.
\end{proof}

\subsection{Short-time Stability}
In this section we establish that for small interval of time, the solution to \eqref{SHE} remains close to its initial function with high probability. This short-time stability is formalized in the following lemma:

\begin{lemma}\label{lem:sup:inf}
	Let $\ell=\lip(f)\vee\lip(g)$.
	Then, there
	exists a universal constant $C_*\ge 4$ such that
	\[
		\P\left\{ \sup_{(0,T]\times\T} w \ge
		(1+\delta)\sup_\T w_0\right\}
		\le C_*\exp\left( -\frac{\delta^2}{C_*\ell^2
		(1+\delta)^2\sqrt{T}}\right),
	\]
	for all $T,\delta\in(0\,,1)$. Moreover, for the same range of variables,
	\[
		 \P\left\{  \inf_{(0,T]\times\T}  w \leq 
		 (1-\delta) \inf_\T w_0 \right\} 
		 \le C_*
		 \exp\left( -\frac{\delta^2}{C_* \ell^2(1+\delta)^2\sqrt{T}}\right). 
	\]
\end{lemma}

Our proof of Lemma \ref{lem:sup:inf}
hinges on the following large-deviations
bound for Walsh stochastic integral processes of the form \eqref{IX}.

\begin{lemma}\label{lem:IX}
	There exists a universal number $C_*\ge 4$ such that
	\[
		\P\left\{ \|I_X\|_{C((0,T]\times\T)}
		\ge y\right\} \le C_*\exp\left( 
		-\frac{y^2}{C_*\sqrt{T}\, 
                \|X\|_{L^\infty((0, T]\times\T\times\Omega)}^2}\right),
	\]
	uniformly for all $y,T>0$, and all predictable space-time random 
	space-time fields $X$ that are uniformly bounded.
\end{lemma}

Lemma \ref{lem:IX} is Lemma 2.1 of Mueller \cite{Mueller1}, except that we have additionally identified parameter
dependencies in that result. This can be done by carefully going through
the proof, the details of which we skip. Instead, we apply Lemma \ref{lem:IX}
in order to verify Lemma \ref{lem:sup:inf} next.

\begin{proof}[Proof of Lemma \ref{lem:sup:inf}: Part I]
	In this part one we establish the first inequality that is claimed in
	Lemma \ref{lem:sup:inf}. Before we begin let us
	observe that, because $C_*\ge 4$, 
	the lemma is nontrivial  only when
	\begin{equation}\label{T:range}
		T^{1/4} \le  \frac{\delta}{16\ell
		(1+\delta)}\wedge 1. 
	\end{equation} 
	Thus, we consider only the above values of $T$.
	
	We begin by studying the special case that $w_0\equiv1$. 
	In that case, let us consider
	the stopping times,
	\[\textstyle
		\tau(R) = \inf\left\{ t>0:\ \sup_{x\in\T} w(t\,,x) \ge R\right\}
		\qquad\forall R>0.
	\]
	Because $w_0\equiv1$, \eqref{eq:mild} and 
	the local properties of the Walsh stochastic integral imply that,
	outside a single $\P$-null set,
	\begin{equation}\label{mild:R}\begin{split}
		w(t\,,x) &\textstyle= 1 + \int_{(0,t)\times\T} p_{t-s}(x\,,y)
			 f_{1+\delta}(w(s\,,y))\,\d s\,\d y\\
		&\textstyle\hskip1in + \int_{(0,t)\times\T} p_{t-s}(x\,,y)
			g_{1+\delta}(w(s\,,y))\,\dot{W}(\d s\,\d y)\\
		&\textstyle= 1 + J_{f_{1+\delta}(w)}(t\,,x) + I_{g_{1+\delta}(w)}(t\,,x),
	\end{split}\end{equation}
	for all $t\in(0\,,\tau(1+\delta)]$ and $x\in\T$, where $f_R(a)=f(a\wedge R)$
	and $g_R(a)= g(a\wedge R)$ for all $a\ge0$, and
	$I_{g_R}$ and $J_{f_R}$ were respectively defined in \eqref{IX}
	and Lemma \ref{lem:Young}.
	
	Lemmas \ref{lem:Young} and \ref{lem:IX} together imply that, uniformly for all 
	$k\ge2$, $T\in(0\,,1)$,
	and $\delta>0$,
	\begin{align}\notag
		&\P\left\{ \tau(1+\delta)\le T\right\}\\\notag
		&\le \P\left\{ \| I_{ g_{1+\delta}(w)}\|_{C((0,T]\times\T)}
			\ge\tfrac{\delta}{2}~,~ \tau(1+\delta)\le T\right\}\\
		&\hskip1in + \P\left\{ \| J_{f_{1+\delta}(w)}\|_{C((0,T]\times\T)} \ge
			\tfrac{\delta}{2} ~,~\tau(1+\delta)\le T\right\}\label{P1}\\\notag
		&\le C_*\exp\left( -\frac{\delta^2}{4C_*M_\delta^2\sqrt{T}}\right)
			+\P\left\{ \|f_{1+\delta}(w)\|_{L^k([0,T]\times\T)} \ge 
			\frac{\delta}{6T^{(2k-3)/(2k)}}\right\},
	\end{align}
	where $C_*\ge 4$ is the same universal constant that appears in Lemma \ref{lem:IX},
	$k\in[2\,,\infty)$ can be as large as we want, and
	\begin{equation}\label{M:delta}
		M_\delta = \sup_{|x|\le1+\delta}| g(x)|\le
		\lip(g)(1+\delta),
	\end{equation}
	because $ g(0)=0$ and $ g $ is Lipschitz continuous.
	Since $f(0)=0$ and $f$ is Lipschitz,
	$\|f_{1+\delta}(w)\|_{C([0, T]\times\T)}\le \lip(f)\{|w|\wedge1+\delta\}\le(1+\delta)\lip(f).$
	In addition, 
	$ \|f_{1+\delta}(w)\|_{L^k([0,T]\times\T)} \leq 2 \|f_{1+\delta}(w)\|_{C([0, T]\times\T)}$
	for all $k\geq 2$ and $T\leq 1$.
	Consequently,
	\begin{equation}\label{P=0}
		\P\left\{ \|f_{1+\delta}(w)\|_{L^k([0,T]\times\T)} \ge 
		\frac{\delta}{6T^{(2k-3)/(2k)}} \right\}=0
	\end{equation}
	as long as
	\[
		T < \left[\frac{\delta}{12(1+\delta)\lip(f)}\right]^{2k/(2k-3)}
		\wedge 1.
	\]
	Owing to \eqref{T:range}, the preceding condition holds when $k=2$.
	This proves that \eqref{P=0} holds when $k=2$, as well. This verifies the first assertion
	of the lemma in the case that $w_0\equiv1$.
	
	In the general case that
	$w_0>0$, consider
	\[
		\tilde{w}(t\,,x) = \frac{w(t\,,x)}{\sup w_0}\qquad\forall t\ge0,\, x\in\T.
	\]
	Then, $\tilde{w}$ satisfies \eqref{eq:mild} with $f$, $g$, and $w_0$
	respectively replaced by 
	\[
		\tilde{f}(z) = \frac{f \left( z\sup_\T w_0\right)}{\sup_\T w_0},\quad
		\tilde{g}(z) = \frac{g \left( z\sup_\T w_0\right)}{\sup_\T w_0},\quad
		\tilde{w}_0(x) = \frac{w_0(x)}{\sup_\T w_0},
	\]
	for all $z\in\R$ and $x\in\T$. 
	Let $\bar{w}$ denote the solution \eqref{SHE} starting
	from $\bar{w}_0\equiv1$ and with $F=(f\,,g)$ replaced by $\tilde{F}=(\tilde{f}\,,\tilde{g})$.
	Because $\tilde{w}_0\le\bar{w}_0$, the comparison theorem
	for  SPDEs (mentioned in the proof of Lemma \ref{lem:dissipation})
	implies that $\P\{\tilde{w}\le\bar{w}\}=1$ and hence, for all $\delta,T>0$,
	\[
		\P\left\{ \adjustlimits\sup_{t\in(0,T]}\sup_{x\in\T}w(t\,,x)\ge (1+\delta)
		\sup_\T w_0\right\}
		\le \P\left\{ \adjustlimits\sup_{t\in(0,T]}\sup_{x\in\T}\bar{w}(t\,,x)\ge 1+\delta\right\}.
	\] 
	Thus, we can apply \eqref{P1} with $(\bar{w}\,,\tilde{ F })$ in place of $(w\,, F )$
	in order to conclude that, for all $\delta,T\in(0\,,1)$
	and $k\ge2$,
	\begin{align*}
		&\textstyle\P\left\{ \sup_{t\in(0,T]}\sup_{x\in\T}w(t\,,x)\ge (1+\delta)
			\sup_\T w_0\right\}\\
		&\textstyle\le C_* \e^{-\delta^2 C_*^{-1} \tilde{M}_\delta^{-2}/\sqrt{T}}
			+\P\left\{ \|\tilde{f}_{1+\delta}(w)\|_{L^k([0,T]\times\T)} \ge 
			\frac{\delta}{6T^{\frac{2k-3}{2k}}} \right\},
	\end{align*}
	where $\tilde{M}$ comes from \eqref{M:delta} but with $g$ replaced
	by $\tilde{g}$; that is,
	\[
		\tilde{M}_\delta = \sup_{|z|\le1+\delta} \frac{%
		 g\left( z\sup_\T w_0\right)}%
		{\sup_\T w_0} \le \lip(g)(1+\delta),
	\]
	once again because $ g(0)=0$ and $g$ is Lipschitz continuous. Because
	$\lip(\tilde{f})\le\lip(f)$, this and \eqref{T:range} together
	yield the first
	assertion of Lemma \ref{lem:sup:inf} in complete generality.
\end{proof}
	
\begin{proof}[Proof of Lemma \ref{lem:sup:inf}: Part II]
	Once again, let us first consider the case that $w_0\equiv1$. In that case, we may
	introduce the stopping times
	\begin{align*}
		\tau_1(\delta) &\textstyle= \inf\left\{ t>0:\, \inf_{x\in\T}w(t\,,x)\le
			1-\delta\right\},\\
		\tau_2(\delta) &\textstyle= \inf\left\{ t>0:\, \sup_{x\in\T}w(t\,,x)\ge 1+\delta\right\}
			\qquad\forall\delta\in(0\,,1).
	\end{align*}
	With the additional constraint ``$w_0\equiv1$'' in mind, we can write
	the following: Uniformly for all $T\in(0\,,1)$, 
	\begin{equation} \label{eq:P-tau-est}
        \begin{split}
		\P&\left\{  \tau_1(\delta) \le T \right\} \le
			\P\left\{\tau_1(\delta)\le T <\tau_2(\delta)\right\} + 
			\P\left\{ \tau_2(\delta)\le T\right\}\\
		&\le \P\left\{\tau_1(\delta)\le T <\tau_2(\delta)\right\} +
			C_*\exp\left(-\frac{\delta^2}{C_*[\lip(g)]^2(1+\delta)^2\sqrt{T}}\right), 
        \end{split}
	\end{equation}
	thanks to first inequality of Lemma \ref{lem:sup:inf}, where $C_*$ is the same universal constant
	that appeared in both Lemma \ref{lem:IX} and the already-established first inequality of
	Lemma \ref{lem:sup:inf}.
	Basic properties of the Walsh integral
	and \eqref{eq:mild} together imply that,
	almost surely on $\{\tau_2(\delta)>T\}$, $w$ solves
	\eqref{mild:R} with $R=1+\delta$,
	where we recall that $f_R(y)=f(y\wedge R)$ and $g_R(y)= g(y\wedge R)$ for all $y\ge0$.
	Recall $I_{g_R}$ and $J_{f_R}$ respectively from  \eqref{IX} and Lemma \ref{lem:Young}
	in order to deduce from \eqref{eq:P-tau-est}, \eqref{P=0}, and from
	Lemma \ref{lem:IX} that
	\begin{align*}
		\P\left\{\tau_1(\delta)\le T <\tau_2(\delta)\right\} 
			&\le \P\left\{ \left\|I_{ g_{1+\delta}(w)}\right\|_{C((0,T]\times\T)} \ge \delta/2\right\}\\
		&\le C_*\exp\left( -\frac{\delta^2}{4C_*\sup_{|z|\le 1+\delta}| g(z)|^2\sqrt{T}}\right)\\
		&	\le C_*\exp\left( -\frac{\delta^2}{C_* [\lip(g)]^2(1+\delta)^2\sqrt{T}}\right),
	\end{align*}
	since $ g(0)=0$ and hence $| g(z)|\le\lip(g)|z|$ for all $z\in\R$ thanks to 
	Assumption \ref{ass:par}.
	This proves Lemma \ref{lem:sup:inf} part 2, under the additional hypothesis that $w_0\equiv1$. In the general case 
	we set
	\[
		\widehat{w}(t\,,x)= \frac{w(t\,,x)}{\inf_\T w_0}\qquad\forall t>0,\, x\in\T,
	\]
	and observe,
	using the comparison theorem for SPDEs (mentioned in the proof of 
        Lemma \ref{lem:dissipation}), that
	$\widehat{w}\ge\underline{w}$ where $\underline{w}$ solves \eqref{SHE}
	starting from $\underline{w}_0\equiv1$
	and with $F\in\{f\,,g\}$ replaced by 
	\[
		\underline{ F }(z) = \frac{F \left( z \inf_\T w_0\right)}{\inf_\T w_0}\qquad
		\forall z\in\R.
	\]
	Thus, the portion that we have already proved shows 
	\begin{align*}
		&\P\left\{  \inf_{(0,T]\times\T}  w \leq 
			(1-\delta) \inf_\T w_0 \right\} =		
			\P\left\{ \inf_{(0,T]\times\T} \widehat{w} \le 1-\delta \right\}\\
		&\le \P\left\{ {\inf_{(0,T]\times\T}} \underline{w} \le 1-\delta \right\}
			\le C_*\exp\left( -\frac{\delta^2}{C_* 
			[\lip(\underline{g})]^2(1+\delta)^2\sqrt{T}}\right).
	\end{align*}
	This inequality yields the desired conclusion since $\lip(\underline{F})=\lip(F)$
	for $F\in\{f\,,g\}$,
	and because $1+\delta<2$.
\end{proof}

\subsection{Local Tightness}
The purpose of this subsection is to highlight the following local
tightness result. This result was first found by Gu and Komorowski \cite{GK}
in the case that $f\equiv0$ and $g(z)\propto z$ for
all $z\in\R$ (the parabolic Anderson model). We outline how one can
obtain the result in the more general case stated below by adapting 
and adjusting their arguments.

\begin{lemma} \label{lem:oscillation}
	For every $a>0$ and $k\in[2\,,\infty)$,
	\begin{equation}\label{mom:oscillation}
		\sup_{t\in[a,1+a]} \E\left(\left|
		\frac{ \sup_{x\in\T} w(t\,,x) }{ \inf_{x\in\T} w(t\,,x)} \right|^k\right)<\infty,
	\end{equation}
	uniformly in the initial data $w_0\in L^\infty(\T)$. 
\end{lemma}

\begin{remark}
	Choose and fix an arbitrary $a>0$.
	Thanks to Lemma \ref{lem:oscillation} and Chebyshev's inequality,
	\begin{equation}\label{eq:oscillation}
		\lim_{b\to\infty} \frac{1}{\log b}\sup_{t\in[a,1+a]}\log
		\P\left\{ \frac{ \sup_{x\in\T} w(t\,,x) }{ \inf_{x\in\T} w(t\,,x)} \ge b \right\}
		= -\infty.
	\end{equation}
	It should be possible to estimate the expectation in \eqref{mom:pos:neg}
	more carefully than will be done below in order to obtain the existence
	of a constant $c=c(a)>1$ such that
	the expectation in \eqref{mom:oscillation} is bounded above
	by $c\exp(ck^3)$, uniformly for all $k\in[2\,,\infty)$.
	This, via a Chernoff-type bound, ought to then yield 
	the following strengthening of \eqref{eq:oscillation}:
	\[
		\limsup_{b\to\infty}
		\frac{1}{(\log b)^{3/2}} \sup_{t\in[a,1+a]} \log
		\P\left\{ \frac{ \sup_{x\in\T} w(t\,,x) }{ \inf_{x\in\T} w(t\,,x)} \ge b \right\}<0.
	\]
	We have not pursued this as we do not need the additional refinement.
\end{remark}

\begin{proof}[Sketch of the proof of Lemma \ref{lem:oscillation}]

	We  follow closely the proof of  Lemma 4.1 of Gu and Komorowski \cite{GK}
	who studied the particular case where $f(z)=0$ and $g(z) \propto z$ for all $z\in\R$ 
	(the parabolic Anderson model). 
	In order to take care of the nonlinear terms $(f\,,g)$, we follow a renormalization
	trick that was used also in our earlier work \cite{KKM2024}.

	According to Mueller \cite{Mueller1}, $w>0$ on $(0\,,\infty)\times\T$ off a single
	$\P$-null set. Define
	\[
		b(t\,, x)=\frac{f(w(t\,, x))}{w(t\,, x)},\quad
		\zeta (t\,, x) : = \frac{g(w(t\,, x))}{w(t\,, x)}\, \dot{W}(t\,, x), 	
	\]
	and consider the SPDE,
	\begin{equation}\left[\label{SHE_Theta_v}\begin{split}
		&\partial_t V = \partial_x^2 V + bV + V\zeta \quad \text{on }(0\,,\infty)\times\T, \\
		&\text{subject to } V(0) = V_0,
	\end{split}\right.\end{equation}
	where $V_0$ is a finite Borel measure on $\T$.
	As was observed in \cite{KKM2024}, we can regard $\zeta$ 
	as a worthy martingale measure whose dominating measure  is bounded below and 
	above by constant multiples of Lebesgue measure. More precisely,
	\[
		M_t(A) = \int_{(0,t)\times A}\zeta(\d s\,\d y)
		= \int_{(0,t)\times A} \frac{g(w(s\,, y))}{w(s\,, y)}\, \dot{W}(\d s\,\d y),
	\]
	defined for every $t\ge0$ and all Borel sets $A\subseteq\T$,
	defines a worthy martingale measure in the sense of Walsh \cite{Walsh}, and
	\[	
		L_1 t |A| \le \<M(A)\>_t = \int_0^t\d s\int_A\d y
		\left| \frac{ g(w(s\,, y)) }{  w(s\,, y)} \right|^2
		\le L_2 t|A|,
	\]
	where $|A|$ denotes the Lebesgue/Haar measure of $A$. The constants $L_1$ and $L_2$
	are inherited from $g$ only; see Assumption \ref{ass:par}. 
	In addition, $|b(t\,,x)|\le\lip(f)$ for all $t\ge0$ and $x\in\T$
	since $f:\R \to \R$ is Lipschitz and $f(0)=0$. 
	Well-established SPDE arguments can be used to show the
	unique solvability of \eqref{SHE_Theta_v}; see 
	Conus, Joseph, Khoshnevisan, and Shiu \cite{CJKS2014}
	and Chen and Dalang \cite{ChenDalang2014}.\footnote{ The results of
	\cite{CJKS2014,ChenDalang2014} are for the drift-free case where $b\equiv0$, though their
	arguments works equally well for the drift terms considered here.}
	In particular, if
	we start the SPDE \eqref{SHE_Theta_v} at $w_0\in L^\infty(\T)$ at time zero, 
	then we recover the solution to \eqref{SHE}
	which also starts at $w_0$.
	
	Now, let $Z_{s,y}(t\,, x)$ denote the fundamental solution to \eqref{SHE_Theta_v}.
	That is, for every choice of $(s\,,y)\in\R_+\times\T$, $Z_{s,y}$ solves the following
linear stochastic heat equation:
	\begin{equation}\left[\label{eq:fundamental}\begin{split}
		&\partial_t Z_{s,y} = \partial^2_x Z_{s,y}+ b\, Z_{s, y} +
			Z_{s,y}\,\zeta\hskip.5in\text{on $(s\,,\infty)\times\T$},\\
		&\text{subject to }Z_{s,y}(s) = \delta_y\hskip1.05in\text{on $\T$}.  
	\end{split}\right.\end{equation}
	As usual, the above is short-hand for the assertion that every $Z_{s,y}$ solves the following 
	SPDE:
	\begin{equation}\label{Z}
	\begin{aligned} 
		Z_{s,y}(t\,,x) &\textstyle
			= p_{t-s}(y-x) +\int_{(s,t)\times\T} p_{t-r}(z-x) b(r\,,z) Z_{s, y}(r\,, z) \d r\, \d z   \\
		&\textstyle\qquad  +\int_{(s,t)\times\T}
			p_{t-r}(z-x) Z_{s,y}(r\,,z)\,\zeta(\d r\,\d z).
	\end{aligned}
	\end{equation}
	Because $\delta_y$ is a finite Borel measure, the preceding paragraphs imply that
	$Z_{s,y}$ exists uniquely for every $s\ge 0$ and $y\in\T$. The following is more than
	enough to take care of forthcoming measurability questions.
	
	\begin{claim}
		The four-parameter process $Z:(t\,,x\,;s\,,y)\mapsto Z_{s,y}(t\,,x)$ is jointly continuous
		on its domain of definition,
		\[
			\text{\rm Dom}(Z) = \left\{(t\,,x\,;s\,,y):\, t>s\ge 0,\ x,y\in\T \right\}.
		\]
	\end{claim}
	Among other things, this Claim can be used to show that \eqref{Z} holds for all
	$(t\,,x\,;s\,,y)\in\text{\rm Dom}(Z)$, off a single $\P$-null set. We leave the (few)
	remaining details to the reader and instead outline the proof of the Claim above.
	
	In order to prove the Claim, it suffices to verify that for every fixed $\varepsilon\in(0\,,1)$, 
	the 4-parameter random field $Z$ is jointly continuous on $\text{\rm Dom}_\varepsilon(Z)$,
	which can be defined as the collection of all points $(t\,,x\,;s\,,y)$ such that
	$t>s+\varepsilon\ge \varepsilon$ and $x,y\in\T$. This is done by
	making an appeal to a suitable formulation of the Kolmogorov
	continuity criterion that works for multiparameter processes -- see
	for an example Garsia \cite{Garsia} -- using the fact that, over the set $\text{\rm Dom}_\varepsilon(Z)$,
	the heat-kernel mapping $(t\,,x\,;s\,,y)\mapsto p_{t-s}(y-x)$ is $C^\infty$, and it is bounded
	with bounded derivatives of all orders  since we are a 
positive distance from the singularity of $p$. We will skip the remaining details
	and leave them to the interested reader. Instead, we move on to make another 
	type of observation.
	
	Since \eqref{SHE_Theta_v}
	is a linear SPDE, a standard appeal to a stochastic Fubini theorem assures us that
	we can superimpose its solution using \eqref{eq:fundamental} as follows:
	\[\textstyle
		w(t\,,x) = \int_\T Z_{t-\nu,y}(t\,,x) w(t-\nu\,,y)\,\d y
		\quad\text{a.s.}\ \forall t\ge\nu, x\in\T.
	\]
	Alternatively, this is a consequence of Markov property, applied at time
	$t-\nu$. It follows that, for all $t\ge\nu$,
	\[
		\frac{\sup_{x\in\T}w(t\,,x) }{\inf_{x\in\T} w(t\,,x)} \leq 
		\frac{\sup_{x\in\T} Z_{t-\nu,y}(t\,, x)}{ \inf_{x, y \in \T} Z_{t-\nu,y}(t\,, x) },
	\]
	which is a random variable thanks to the Claim above.
	Choose and fix $t\ge\nu>0$ as above.
	Thanks to the translation invariance of $Z$, itself inherited from the translation invariance of
	the white noise $\dot{W}$,
	the law of the spatial random field $(x\,,y)\mapsto Z_{t-\nu,y}(t\,,x)$ does not
	depend on the numerical value of $t\ge\nu$. 
	In light of H\"older's inequality, 
	it then suffices to prove that
	\begin{equation}\label{mom:pos:neg} \textstyle
		\E\left(\sup_{x, y \in \T} \left[ Z_{0,y}(T, x) \right]^k\right)<\infty
		\quad\forall T>0,\, k\in\Z.
	\end{equation}
	[N.B. The above holds for all $k\in\Z$ and not only $k\in\N$.]
	
	We now consider the following SPDE: 
	\begin{equation}\left[\label{eq:fundamental2}\begin{split}
		&\partial_t \tilde Z_{s,y} = \partial^2_x \tilde Z_{s,y} +
			\tilde Z_{s,y}\,\zeta\hskip1in\text{on $(s\,,\infty)\times\T$},\\
		&\text{subject to }\tilde Z_{s,y}(s) = \delta_y\hskip1.05in\text{on $\T$}. 
	\end{split}\right.\end{equation}
	In the case where  $f(z)=0$ and $g(z) \equiv z$, 
	Lemma 4.1 of Yu Gu and Komorowski
	\cite{GK} asserts precisely the above moment bound 
	\eqref{mom:pos:neg}. Close inspection of those arguments shows 
	the  said bounds in \cite{GK} depend primarily 
	on the fact that the space-time white noise
	$\dot{W}$ is a worthy martingale measure whose dominating measure is the Lebesgue (or Haar) measure. 
	In other words, the Gaussian nature of the noise does not play a key role in the analysis of Lemma 4.1
	of \cite{GK}.  In the present more general setting where $z\mapsto g(z)$ can be  nonlinear, 
	the dominating measure for $\zeta$ is to within a
	constant multiple of the Lebesgue measure. And that is enough to ensure
	that the proof of Lemma 4.1 of \cite{GK} still works for $\tilde Z_{s, y}$ 
	in \eqref{eq:fundamental2}. In other words,
	\begin{equation}\label{mom:pos:neg2}\textstyle
		\E\left(\sup_{x, y \in \T}  [\tilde{Z}_{0,y}(T, x)  ]^k\right)<\infty
		\quad\forall T>0,\, k\in\Z.
	\end{equation}
	
	Now we consider the general case,  and define two space-time random fields $Z^+$ and $Z^-$ 
	as follows:
	\begin{equation}\label{fundamental34}
		Z^\pm_{s, y}(t\,, x) : = \tilde Z_{s, y}(t\,, x) \e^{\pm\lip(f) (t-s)}.
	\end{equation} 
	Thanks to \eqref{mom:pos:neg2} and \eqref{fundamental34},
	\begin{equation}\label{mom:pos:neg3}\textstyle
		\E\left(\sup_{x, y \in \T} [ Z^\pm_{0,y}(T\,, x) ]^k\right)<\infty
		\quad\forall T>0,\, k\in\Z.
	\end{equation}
	In addition, it is clear that  $Z^\pm$ satisfy the following SPDEs: 
	\begin{align*}
		&\partial_t Z^\pm_{s,y} = \partial^2_x Z^\pm_{s,y}
			\pm\lip(f) Z^\pm_{s, y} + 
			Z^\pm_{s,y}\,\zeta\quad\text{on $(s\,,\infty)\times\T$},\\
		&\text{subject to } Z^\pm_{s,y}(s) = \delta_y\hskip1.05in\text{on $\T$}.
	\end{align*}
	Because $-\lip(f)  \le z^{-1}f(z)  \le \lip(f)$ for all $z> 0$, and since
	$Z_{0, y},Z^\pm_{0, y}\ge0$,  the comparison argument
        (mentioned in the proof of Lemma \ref{lem:dissipation}) shows that 
	$Z^-_{0, y} \le Z_{0, y}\le Z^+_{0, y}$ for all $y\in\T$.
	Therefore, \eqref{mom:pos:neg3} implies \eqref{mom:pos:neg}
	and concludes the proof.
\end{proof}

We conclude this section with the following corollary of Lemma \ref{lem:oscillation}.

\begin{proposition} \label{pr:oscillation}
	For every $k\in[2\,,\infty)$,
	\[
		\sup_{t\geq 1} \E\left(\left|
		\frac{ \sup_{x\in\T} w(t\,,x) }{ \inf_{x\in\T} w(t\,,x)} \right|^k\right)<\infty,
	\]
	uniformly in the initial data $w_0\in L^\infty(\T)$.
\end{proposition}

\begin{proof}
	Fix $k\ge2$ as instructed,  define
	\[
		\Lambda = \adjustlimits
		\sup_{t\in[1\,,2]} \sup_{w_0\in L^\infty(\T)}\E\left(\left|
		\frac{ \sup_{x\in\T} w(t\,,x) }{ \inf_{x\in\T} w(t\,,x)} \right|^k\right),
	\]
	and apply the Markov inequality in order to see that
	\[
		\sup_{t\in[n,n+1]} \E\left(\left.\left|
		\frac{ \sup_{x\in\T} w(t\,,x) }{ \inf_{x\in\T} w(t\,,x)} \right|^k\ \right| \,
		\mathscr{F}(n)\right)\le\Lambda.
	\]
	Optimize over $n\in\Z_+$ to finish.
\end{proof}

\subsection{Tracking}\label{subsec:track}
Here and throughout let us keep in mind that 
\begin{equation}\label{mu:<}
	-\infty < \mu < {\rm L}_g^2/64.
\end{equation}
In order to see this, let us recall
that $\mu=f'(0+)$ [see \eqref{mu:sigma}] and $f(0)=0$
[see Assumption \ref{ass:par}], and hence
$\mu = \lim_{z\to0+}[f(z)/z]\le \sup_{z>0}[f(z)/z]$, which
yields \eqref{mu:<} thanks to \eqref{cond:f}. Let us introduce
a new variable $\gamma=\gamma(f\,,g)$, that satisfies $\gamma>0$
thanks to \eqref{cond:f} and is defined as
\begin{equation}\label{gamma}\textstyle
	\gamma = \frac{{\rm L}_g^2}{64}-\sup_{z>0}\frac{f(z)}{z}
\end{equation}
Let us pause to also define
\begin{equation}\label{tau(T)}
	\tau(T) = \inf\left\{ t \ge T:\  \|w(t)\|_{C(\T)} > \exp(-\gamma t)\right\},
\end{equation}
where $\inf\varnothing=\infty$ and $\gamma$ was defined in \eqref{gamma}. 
We plan to use the fact that $\tau(T)=\infty$ with high probability when $T\gg1$.
Specifically (Lemma \ref{lem:dissipation}), that there exists
a number $L_* =L_*(f,g)>1$ such that
\begin{equation}\label{P(tau)}
	\P\{\tau(T) = \infty\} \ge 1 -  L_*\e^{-T/L_*}
	\qquad\forall T>0.
\end{equation}

With the preceding in place, let us recall the notation in \eqref{R:star} --
to be used for the linearization errors for $f$ and $g$ -- in order to state the main
estimate of this subsection. The goal of this subsection is to establish the
next result which shows that if we start \eqref{SHE} and \eqref{PAM} at the
same place at a large time $T$, then with high probability,
the solutions track one another very closely for a short time.
       
\begin{proposition}\label{pr:tracking}
	For every $T>0$ let $u=u^{(T)}$ denote the solution to the parabolic Anderson model,
	\begin{equation*}\left[\begin{split}
		&\partial_t u = \partial^2_x u + \mu u  + \sigma u\dot{W}
			\hskip.95in\text{on $(T\,,\infty)\times\T$},\\
		&\text{subject to }u(0)=w(T)\hskip.9in\text{on $\T$}.
	\end{split}\right.\end{equation*}
	Then, for every $k \in (6\,,\infty)$ there exist non-random numbers
	$L=L(k\,,f\,,g)>0$ and 
	$\Delta_0=\Delta_0(k\,,f\,,g)\in(0\,,1)$ such that, independently of
	the value of $T$ and uniformly for all $\Delta(T)\in(0\,,\Delta_0)$,
	\[
		\E_{\mathscr{F}(T)}\left(\|w_T-u_T\|_{C([T,T+\Delta(T)]\times\T)}^k
		\right)\label{w_T-u_T}
		\le L \left[\mathscr{E}\left(\e^{-\gamma T}\right)\right]^k\|w(T)\|_{C(\T)}^k,
	\]
	almost surely on the event $\{\tau(T)>T\}$, where
	$\mathscr{E}$ was defined in \eqref{R:star}, and
	\[
		w_T(t\,,x) = w(t\wedge\tau(T)\,,x),\
		u_T(t\,,x) = u(t\wedge\tau(T)\,,x)\quad\forall t\ge T,\ x\in\T.
	\] 
\end{proposition}

\begin{proof}[Proof of Proposition \ref{pr:tracking}]
	Next, we recall that $u$ is the solution to the following (mild) integral equation:
	\begin{align*}
		u(t\,,x) &\textstyle= (p_t*w(T))(x) + \mu \int_{(T,\,t)\times\T} p_{t-s}(x\,,y) u(s\,,y)\,\d s\,\d y\\
		&\textstyle\hskip3cm + \sigma\int_{(T,\,t)\times\T} p_{t-s}(x\,,y) u(s\,,y)\,\dot{W}(\d s\,\d y).
	\end{align*}
	For every $z\in\R\setminus\{0\}$ define
	$\mathscr{E}_f(z) =[f(z)-\mu z]/z$ and
	$\mathscr{E}_g(z) = [g(z)-\sigma z]/z$
	in order to be able to write
	\begin{equation}\label{track}\textstyle
		w_T(t\,,x) - u_T(t\,,x) =\sum_{j=1}^4 H_j(t\,,x)
		\qquad\forall t\in[T\,,T+\Delta(T)],
	\end{equation}
	where
	\begin{align*}
		H_1(t\,,x) &\textstyle = \mu \int_{(T,\,t\wedge\tau(T))\times\T} 
			p_{(t\wedge\tau(T))-s}(x\,,y)\left[ w(s\,,y) -  u(s\,,y)\right]\d s\,\d y,\\
		H_2(t\,,x) &\textstyle = \int_{(T,\,t\wedge\tau(T))\times\T} p_{(t\wedge\tau(T))-s}(x\,,y)
			w(s\,,y)\mathscr{E}_f(w(s\,,y))\,\d s\,\d y,\\
		H_3(t\,,x) &\textstyle = \sigma \int_{(T,\,t\wedge\tau(T))\times\T} 
			p_{(t\wedge\tau(T))-s}(x\,,y)\left[ w(s\,,y) -  u(s\,,y)\right]
			\dot{W}(\d s\,\d y),\\
		H_4(t\,,x) &\textstyle = \int_{(T,\,t\wedge\tau(T))\times\T} p_{(t\wedge\tau(T))-s}(x\,,y)
			w(s\,,y)\mathscr{E}_g(w(s\,,y))\,\dot{W}(\d s\,\d y).
	\end{align*} 
	
	Equivalently,
	\begin{align*}
		H_1(t\,,x) &\textstyle = \mu \int_{(T,\,t\wedge\tau(T))\times\T} 
			p_{(t\wedge\tau(T))-s}(x\,,y)\left[ w_T(s\,,y) -  u_T(s\,,y)\right]\d s\,\d y,\\
		H_2(t\,,x) &\textstyle = \int_{(T,\,t\wedge\tau(T))\times\T} p_{(t\wedge\tau(T))-s}(x\,,y)
			w_T(s\,,y)\mathscr{E}_f(w_T(s\,,y))\,\d s\,\d y,\\
		H_3(t\,,x) &\textstyle = \sigma \int_{(T,\,t\wedge\tau(T))\times\T} 
			p_{(t\wedge\tau(T))-s}(x\,,y)\left[ w_T(s\,,y) -  u_T(s\,,y)\right]
			\dot{W}(\d s\,\d y),\\
		H_4(t\,,x) &\textstyle = \int_{(T,\,t\wedge\tau(T))\times\T} p_{(t\wedge\tau(T))-s}(x\,,y)
			w_T(s\,,y)\mathscr{E}_g(w_T(s\,,y))\,\dot{W}(\d s\,\d y).
	\end{align*}
	We analyze the random fields $H_1,\ldots,H_4$ separately.
        
         Since $\Delta(T)\le 1$, we may apply Lemma \ref{lem:Young}  in order to see
        that for all real numbers $k>6$, and independently of the value of $T>0$,
	\begin{align}\label{track:I_1}
		&\textstyle\E_{\mathscr{F}(T)}\left( \sup_{t\in[T,T+\Delta(T)]}
			\| H_1(t)\|_{C(\T)}^k\right)\\ \nonumber
		&\hskip1in\lesssim  [\Delta(T)]^{k(2k-3)/(2k)}\E_{\mathscr{F}(T)}\left(
			\|w_T - u_T\|_{L^k([T,T+\Delta(T)]\times\T)}^k\right)\\ \nonumber
		&\hskip1in \lesssim  [\Delta(T)]^{k-(1/2)} \E_{\mathscr{F}(T)}\left(
			\|w_T - u_T\|_{C([T,T+\Delta(T)]\times\T)}^k\right), 
	\end{align}
	almost surely on $\{\|w(T)\|_{C(\T)} \le \exp(-\gamma T)\}$, 
	where the implied constant is nonrandom and depends only on $(k\,,\mu)$. 
	This is the desired estimate for $H_1$. 
	
	We begin our analysis of $H_2$ in like manner. Namely, we observe that, a.s.\
	on $\{\|w(T)\|_{C(\T)} \le \exp(-\gamma T)\}$,
	\begin{align}
		&\textstyle\textstyle\E_{\mathscr{F}(T)}\left( \sup_{t\in[T,T+\Delta(T)]}
			\| H_2(t)\|_{C(\T)}^k \right)
			\label{track:pre:I_2}\\\nonumber
		&\textstyle\lesssim  [\Delta(T)]^{k(2k-3)/(2k)} \E_{\mathscr{F}(T)}\left(
			\|w_T(\mathscr{E}_f\circ w_T)\|_{L^k([T,T+\Delta(T)]\times\T)}^k\right)\\\nonumber
		&\textstyle\le  [\Delta(T)]^{k-(3/2)} \E_{\mathscr{F}(T)}\left(
			\|w_T(\mathscr{E}_f\circ w_T)\|_{L^k([T,T+\Delta(T)]\times\T)}^k \right),
	\end{align}
	where ``$\circ$'' denotes as usual the composition of functions, and
	the implied constant is nonrandom and  depends only on $(k\,,\mu)$. 
	By virtue of definition, the following holds
	$\omega$-by-$\omega$:
	\begin{equation}\label{w_T:exp}\textstyle
		\sup_{t\in[T,T+\Delta(T)]}\sup_{x\in\T} w_T (t\,,x)\le \e^{-\gamma T}.
	\end{equation}
	It follows from \eqref{track:pre:I_2} and the monotonicity of $\mathscr{E}$ --
	see \eqref{R:star} --  that, almost surely on 
	$\{\|w(T)\|_{C(\T)} \le \exp(-\gamma T)\}$,
	\begin{align*}
		&\textstyle\textstyle\E_{\mathscr{F}(T)}\left( \sup_{t\in[T,T+\Delta(T)]}
			\| H_2(t)\|_{C(\T)}^k \right)\\
		&\textstyle\lesssim  [\Delta(T)]^{k-(3/2)}
			\left[ \mathscr{E}\left(\e^{-\gamma T}\right)\right]^k
			\E_{\mathscr{F}(T)}\left(
			\|w_T\|_{L^k([T,T+\Delta(T)]\times\T)}^k \right)\\
		&\textstyle\le 2 [\Delta(T)]^{k-(1/2)} \left[ \mathscr{E}\left(\e^{-\gamma T}\right)\right]^k
			\sup_{(t,x)\in[T,T+\Delta(T)]\times\T}
			\E_{\mathscr{F}(T)}\left(
			|w_T(t\,,x)|^k\right).
	\end{align*}
	It is a standard part of SPDEs that the process
	$\{w(t)\}_{t\ge0}$ associated to \eqref{SHE} is a strong Markov process
	with values in $C(\T)$, and that there exists a constant $c=c(f,g)>0$,
	independently of the choice of $(k\,,w_0)\in[2\,,\infty)\times L^\infty(\T)$, such that
	\begin{equation}\label{track:mom}\textstyle
		\E\left( \|w\|_{C([0,1]\times\T)}^k\right) \le c^k\|w_0\|_{C(\T)}^k
		\qquad\forall (k\,,w_0)\in[2\,,\infty)\times C(\T).
	\end{equation}
	Since $\Delta(T)\le 1$ and $w$ is continuous away from $t=0$
	as long as $w_0\in L^\infty(\T)$, it follows from this and an application of
	the Markov property at time $T$ that a.s.\ on $\{\|w(T)\|_{C(\T)} \le \exp(-\gamma T)\}$,
	\[\textstyle
		\sup_{(t,x)\in[T,T+\Delta(T)]\times\T}
		\E_{\mathscr{F}(T)}\left(
		|w_T(t\,,x)|^k\right) \le c^k\|w(T)\|_{C(\T)}^k
		\quad\text{a.s.,}
	\]
	uniformly for all $k\in[2\,,\infty)$,
	and for the same constant $c$ as above. Because $c$ is deterministic and depends only
	on $f$ and $g$, it follows that there exists a deterministic number $c_1=c_1(f,g)>0$,
	independently of all else, such that almost surely on $\{\|w(T)\|_{C(\T)} \le \exp(-\gamma T)\}$,
	\begin{equation}\label{track:I_2}\begin{split}
		&\textstyle\E_{\mathscr{F}(T)}\left( \sup_{t\in[T,T+\Delta(T)]}
			\| H_2(t)\|_{C(\T)}^k \right)\\
		&\textstyle\hskip1in\le \left[c_1  \mathscr{E}\left(\e^{-\gamma T}\right)
			\|w(T)\|_{C(\T)}\right]^k  [\Delta(T)]^{k-(1/2)},
	\end{split}\end{equation}
	uniformly for all $k\in[2\,,\infty)$.
	This is the desired estimate for $H_2$. Now we proceed to analyze $H_3$.
	
	Next we observe that
	\[\textstyle
		\E_{\mathscr{F}(T)}\left( \sup_{t\in[T,T+\Delta(T)]}
		\| H_3(t)\|_{C(\T)}^k \right) \le \sigma^k
		\E_{\mathscr{F}(T)}(\bm{\mathcal{T}}^k),
	\]
	valid a.s.\ on $\{\|w_T\|_{C(\T)} \le \exp(-\gamma T)\}$, where 
	\[\textstyle
		\bm{\mathcal{T}} = \sup_{\substack{t\in[T,T+\Delta(T)]\\x\in\T}}
		\left|   \int_{(T,\,t)\times\T} p_{(t\wedge\tau(T))-s}(x\,,y)\left[ 
		w_T(s\,,y) -  u_T(s\,,y)\right]
		\dot{W}(\d s\,\d y)\right|
	\]
	Therefore, Lemma \ref{lem:Salins} yields
	\begin{align}\nonumber\textstyle
		&\textstyle\E_{\mathscr{F}(T)}\left( \sup_{t\in[T,T+\Delta(T)]}
			\| H_3(t)\|_{C(\T)}^k \right)\\\nonumber
		&\textstyle\quad\lesssim[\Delta(T)]^{(k-6)/4}\int_T^{T+\Delta(T)}\d s\int_\T\d y\
			\E_{\mathscr{F}(T)}\left( \left|  w_T(s\,,y) -  
			u_T(s\,,y)\right|^k\right)\\
		&\textstyle\quad\le 2[\Delta(T)]^{(k-2)/4}\E_{\mathscr{F}(T)}\left( 
			\| w_T - u_T \|_{C([T,T+\Delta(T)]\times\T)}^k \right),
			\label{track:I_3}
	\end{align}
	valid a.s.\ on $\{\|w(T)\|_{C(\T)} \le \exp(-\gamma T)\}$,
	where the implied constant is nonrandom and depends only on $(k\,,f\,,g)$.
	An additional appeal to Lemma \ref{lem:Salins} yields the following estimates
	for $H_4$, valid a.s. on $\{\|w(T)\|_{C(\T)}\le\exp(-\gamma T)\}$:
	\begin{align*}\textstyle
		\E_{\mathscr{F}(T)}&\textstyle\left( \sup_{t\in[T,T+\Delta(T)]}
			\| H_4(t)\|_{C(\T)}^k \right)\\
		&\lesssim  [\Delta(T)]^{(k-6)/4} \E_{\mathscr{F}(T)}\left(
			\|w_T(\mathscr{E}_g\circ w_T)\|_{L^k([T,T+\Delta(T)]\times\T)}^k\right)\\\nonumber
		&\le  2[\Delta(T)]^{(k-2)/4} \E_{\mathscr{F}(T)}\left(
			\|w_T(\mathscr{E}_g\circ w_T)\|_{C([T,T+\Delta(T)]\times\T)}^k \right),
	\end{align*}
	where the implied constant is nonrandom and depends only on $(k\,,f\,,g)$.
	Owing to \eqref{w_T:exp}, the preceding development
	yields the following, valid a.s.\
	on $\{\|w(T)\|_{C(\T)}\le\exp(-\gamma T)\}$:
	\begin{align*}
		\E_{\mathscr{F}(T)}&\textstyle\left( \sup_{t\in[T,T+\Delta(T)]}
			\| H_4(t)\|_{C(\T)}^k \right)\\
		&\lesssim [\Delta(T)]^{(k-2)/4} \left[\mathscr{E}\left( \e^{-\gamma T}\right)\right]^k
			\E_{\mathscr{F}(T)}\left(
			\|w_T\|_{C([T,T+\Delta(T)]\times\T)}^k \right),
	\end{align*}
	where the implied constant is nonrandom and depends only on $(k\,,f\,,g)$. Therefore,
	\eqref{track:mom} and the Markov property applied at time $T$ together yields
	the following: A.s.\ on $\{\|w(T)\|_{C(\T)}\le\exp(-\gamma T)\}$,
	\begin{align}
		&\textstyle\E_{\mathscr{F}(T)}\left( \sup_{t\in[T,T+\Delta(T)]}
			\| H_4(t)\|_{C(\T)}^k \right)
			\label{track:I_4}\\\nonumber
		&\hskip1in\lesssim [\Delta(T)]^{(k-2)/4} \left[\mathscr{E}\left( \e^{-\gamma T}\right) 
			\|w(T)\|_{C(\T)}\right]^k,
	\end{align}
	where once again the implied constant is nonrandom and depends only on $(k\,,f\,,g)$. 
	Combine \eqref{track},
	\eqref{track:I_1}, \eqref{track:I_2}, \eqref{track:I_3}, and \eqref{track:I_4},
	and recall that $\Delta(T)\le1$ and $k>6$,
	in order to see that almost surely on the event $\{\|w(T)\|_{C(\T)}\le\exp(-\gamma T)\}$,
	\begin{equation*}
        \begin{split}
		&\E_{\mathscr{F}(T)}\left( \|w_T-u_T\|_{C([T,T+\Delta(T)]\times\T)}^k\right)\\
		&\quad\lesssim [\Delta(T)]^{(k-2)/4} \E_{\mathscr{F}(T)}\left(
			\|w_T - u_T\|_{C([T,T+\Delta(T)]\times\T)}^k\right)\\
		&\qquad+[\Delta(T)]^{(k-2)/4} 
			\left[\mathscr{E}\left(\e^{-\gamma T}\right)\right]^k
			\|w(T)\|_{C(\T)}^k,
        \end{split}
	\end{equation*}
	where the implied constant is nonrandom and depends only on $(k\,,f\,,g)$. 
	Note in particular that the implied constant does not depend on $T$ nor on $\Delta(T)$. Therefore, 
	for every $k>6$ there exist non-random numbers $L=L(k\,,f\,,g)>0$
	and $\Delta_0=\Delta_0(k\,,f\,,g)\in(0\,,1)$
	such that for every $T>0$ and $\Delta(T)\in(0\,,\Delta_0)$,
	\begin{equation*}
        \begin{split}
		&\E_{\mathscr{F}(T)}\left( \|w_T-u_T\|_{C([T,T+\Delta(T)]\times\T)}^k\right)\\
		&\textstyle\quad\le \frac12 \E_{\mathscr{F}(T)}\left(
			\|w_T - u_T\|_{C([T,T+\Delta(T)]\times\T)}^k\right)
			+ \frac{L}{2}\left[
			\mathscr{E}\left( \e^{-\gamma T}\right)\right]^k
			\|w(T)\|_{C(\T)}^k,
        \end{split}
	\end{equation*}
	almost surely on the event $\{\|w(T)\|_{C(\T)}\le\exp(-\gamma T)\}$.
	This has the desired result.
\end{proof}

\section{A Coupling Estimate for PAM}\label{sec:coupling}
The goal of this section is to recall, and say a few things, about a coupling method of Mueller \cite{Mueller2}
that was dubbed as AM/PM coupling in \cite{KKM2023a}.

Let $u$ denote the solution to the parabolic Anderson model 
\begin{equation}\label{PAM:u}\left[\begin{split}
	&\partial_t u = \partial^2_x u + \mu u + \sigma u \dot{W}\hskip.5in\text{on $(0\,,\infty)\times\T$},\\
	&\text{subject to }u=u_0\hskip0.8in\text{on $\T$}.
\end{split}\right.\end{equation}
Throughout this section, we choose and fix a number $\alpha \in (0\,, 1)$ and define for all $y\in\R$,
\begin{equation}\label{Phi:W}\textstyle
	\Phi_\alpha(y) = \sqrt{(\alpha|y|\wedge1)},\
	\Psi_\alpha(y) = \sqrt{1-|\Phi_\alpha(y)|^2} = \sqrt{1-(\alpha|y|\wedge 1)},
\end{equation}
and let $\dot{W}_0$ denote a space-time white noise that is independent of $\dot{W}$. Then, define
\begin{equation}\label{PAM:v}\left[\begin{split}
	&\textstyle\partial_t v = \partial^2_x v + \mu v + \sigma v\left\{ \Psi_\alpha \left(\frac{u-v}{v}\right) \dot{W}+
        		\Phi_\alpha \left(\frac{u-v}{v}\right)\dot{W}_0\right\},\\
	&\textstyle\text{\text{on $(0\,,\infty)\times\T$}, subject to }v=v_0\text{ on $\T$}.
\end{split}\right.\end{equation}
It follows from a well-known tightness argument that there exists a pairing of the random field
$v=\{v(t\,,x)\}_{t\ge0,x\in\T}$ and a white noise $\dot{W}$ that together solve
\eqref{PAM:v} -- possibly after we enlarge the underlying probability 
space. The strong uniqueness  of $v$ does not hold; see Mueller, Mytnik, and Perkins \cite{MMP} and
it is not known whether strong uniqueness holds among positive solutions.
Still, one can see immediately that the law of $v$ is unique and is the same as the law of the solution to \eqref{PAM}, 
when $u_0$ in the latter is replaced by $v_0$ in the former. This is because [the martingale measure 
defined by]
\[\textstyle
	\Psi_\alpha \left(\frac{u-v}{v}\right) \dot{W}+
        \Phi_\alpha\left(\frac{u-v}{v}\right)\dot{W}_0,
\]
is in fact a space-time white noise\footnote{See the appendix of 
\cite{KKM2023a} for a proof. This uses the property 
that $\Psi^2+\Phi^2=1$, pointwise.}. Therefore, it follows that $v$ has the law of a parabolic Anderson model,
the same as $u$ in \eqref{PAM:u} but with a possibly different initial profile $v_0$, as was stated a moment ago.

\begin{proposition}\label{pr:coupling}
	There exists a number
	$A=A(f\,,g)>1$ such that
	\begin{align*}
		&\P\left\{ u(s)\neq v(s)\text{ for all $s\in(0\,,t]$}\right\}\\
		&\hskip.5in\le A\, \left( \frac{\|u_0 - v_0\|_{C(\T)}}{%
			\alpha t \min\left( \inf_{z\in \T} u_0(z)  \,, \inf_{z\in \T} v_0(z) \right)}\right)^{1/2}
			+ A\, \exp\left( -\frac{1}{A\sqrt{t}}\right),
	\end{align*}
	uniformly for all $t>0$ and all bounded, measurable, and non-zero 
	respective initial profiles $u_0,v_0:\T\to\R_+$.
\end{proposition}

This is the main result of this section.
Before we prove it, let us pause to make the following simple observation.

\begin{lemma}\label{lem:elem}
	If $p,q>0$ satisfy  $|p-q|\le p\wedge q$, then
	 $p\vee q\le 2(p\wedge q)$.
\end{lemma}

We now return to Proposition \ref{pr:coupling} and observe that,
because $A>1$, the proposition has content only when
\begin{equation}\label{tu0}
	t<1
	\quad\text{and}\quad
	\|u_0-v_0\|_{C(\T)}< t\min\left(  \inf_{z\in \T} u_0(z)  \,, \inf_{z\in \T} v_0(z) \right).
\end{equation}
Thus, we see that Proposition \ref{pr:coupling}, Lemma \ref{lem:elem},
and \eqref{tu0} imply, and hence is equivalent to, the following. 

\begin{proposition}\label{pr:coupling:1}
	We can find 
	a number $A=A(f\,,g)>1$ such that
	\begin{align*}
		&\P\left\{ u(s)\neq v(s)\text{ for all $s\in(0\,,t]$}\right\}\\
		&\hskip.5in\le A\, \left( \frac{\|u_0 - v_0\|_{C(\T)}}{%
			\alpha t\max\left( \inf_{z\in \T} u_0(z)  \,, \inf_{z\in \T} v_0(z) \right) }\right)^{1/2}
			+ A\, \exp\left( -\frac{1}{A\sqrt{t}}\right),
	\end{align*}
	uniformly for all $t>0$ and all bounded, measurable, and non-zero 
	respective initial profiles $u_0,v_0:\T\to\R_+$.
\end{proposition} 
If it is more convenient, we can replace $\max(\inf_{z\in \T} u_0(z)\,,\inf_{z\in \T} v_0(z))$ by either $\inf_{z\in \T} u_0(z)$ or  $\inf_{z\in \T} v_0(z)$ in Proposition \ref{pr:coupling:1}. 
In other words, if $\|u_0-v_0\|_{C(\T)} \ll t \inf_{z\in \T} u_0(z)$
for some $t\ll1$, then there is a coupling of two solutions $u$ and $v$ to the parabolic 
Anderson models, starting respectively at $u_0$ and $v_0$, such that
the coupling succeeds  at some point in the time interval 
$[0\,,t]$ with very high probability. Basic properties of these types of SPDEs 
(the Markov property and the stability of solutions) show that once the
coupling of $u$ and $v$ succeeds, say at time $t$, then the two solutions
are equal from then on; that is, $u(s)= v(s)$ for all $s\ge t$ a.s.\ on
the event $\{u(t)=v(t)\}$.

\begin{proof}
	First of all, we can and will assume without loss in generality that
	\begin{equation}\label{mu=0}
		\mu=0.
	\end{equation}
	For otherwise we could instead study $\tilde{u}(t\,,x)=\exp(-\mu t)u(t\,,x)$
	and $\tilde{v}(t\,,x)=\exp(-\mu t)v(t\,,x)$. The random fields $\tilde{u}$
	and $\tilde{v}$ solve the same SPDEs as do $u$ and $v$, respectively, but
	with $\mu$ replaced by $0$ everywhere. Since $u(t)=v(t)$ if and only if
	$\tilde{u}(t)=\tilde{v}(t)$ this justifies the assumption \eqref{mu=0},
	which we enforce from here on, and without further mention.
	
	Let $\bar{u}$ satisfy the parabolic Anderson model,
	\[
		\partial_t \bar{u} = \partial^2_x\bar{u} + \sigma\bar{u}\dot{W}\quad\text{on $(0\,,\infty)\times\T$},
	\]
	subject to $\bar{u}(0)= u_0\vee v_0$, and $\underline{u}$ solve the parabolic Anderson model
	\[\textstyle
		\partial_t \underline{u} = \partial^2_x\underline{u} + \sigma\underline{u}
		\left[\Psi_\alpha \left(\frac{\bar{u}-\underline{u}}{\underline{u}}\right) \dot{W}
		+\Phi_\alpha \left(\frac{\bar{u}-\underline{u}}{\underline{u}}\right)\dot{W}_0
		\right]\qquad\text{on $(0\,,\infty)\times\T$},
	\]
	subject to $\underline{u}(0)= u_0\wedge v_0$. 
	The comparison theorem for SPDEs (mentioned in the proof of Lemma \ref{lem:dissipation})
	implies that
	\begin{equation}\label{comparison:uv}
		\underline{u}\le u\wedge v\le u\vee v\le \bar{u},
	\end{equation}
	off a single $\P$-null set. Note that if $\underline{u}(t)=\overline{u}(t)$ for some $t>0$,
	then certainly $u(t)=v(t)$. Therefore, the fact that $|(u_0\wedge v_0) - (u_0\vee v_0)|=|v_0-u_0|$
	allows us to assume without loss of generality that 
	\begin{equation}\label{u_0<v_0}
		u_0 \le v_0,
	\end{equation}
	for otherwise we can replace the quadruple $(u_0\,,v_0\,,u\,,v)$ by the quadruple
	$(u_0\wedge v_0\,,u_0\vee v_0\,,\underline{u}\,,\bar{u})$ everywhere, and systematically, in the remainder
	of this proof in order to derive the same conclusion. We assume \eqref{u_0<v_0} henceforth. Owing to the 
	comparison theorem for SPDEs -- see \eqref{comparison:uv} -- it follows also that
	\begin{equation}\label{D>0}
		D = v - u \ge 0\qquad\text{off a single $\P$-null set.}
	\end{equation}
	One can check directly that the random field $D=\{D(t\,,x)\}_{t\ge0,x\in\T}$ satisfies the SPDE
	\[
		\partial_t D = \partial^2_x D + \sigma h_\alpha(u\,,v)\zeta
		\qquad\text{on }(0\,,\infty)\times\T,
	\]
	where $\zeta$ is another space-time white noise, and 
	\begin{equation}\label{h^2}\textstyle
		| h_\alpha(u\,,v)|^2= v^2 \left|\Phi_\alpha\left( \frac{u-v}{v}\right)\right|^2 
		+ \left| v\Psi_\alpha \left(\frac{u-v}{v}\right) -u\right|^2.
	\end{equation}
	To be sure, we remind that the preceding means that $D$ solves the 
	integral equation,
	\[\textstyle
		D(t\,,x) = (p_t*(v_0-u_0))(x) + \sigma \int_{(0,t)\times\T}
		p_{t-s}(x\,,y) h_\alpha (u(s\,,y)\,,v(s\,,y))\,\zeta(\d s\,\d y).
	\]
	In particular, it follows from a stochastic Fubini argument and
	\eqref{u_0<v_0} that
	\begin{equation}\label{X}\textstyle
		X(t) = \int_\T D(t\,,x)\,\d x = \|v(t)-u(t)\|_{L^1(\T)} \qquad[t\ge0]
	\end{equation}
	satisfies $X(0) = \|v_0-u_0\|_{L^1(\T)}$ and
	\[\textstyle
		X(t) = \|v_0-u_0\|_{L^1(\T)} + \sigma \int_{(0,t)\times\T}
		h_\alpha (u(s\,,y)\,,v(s\,,y))\,\zeta(\d s\,\d y)
		\qquad\forall t>0.
	\]
	Consequently, $X=\{X(t)\}_{t\ge0}$ is a continuous $L^2$-martingale
	with mean $\|v_0-u_0\|_{L^1(\T)}$ and quadratic variation,
	\begin{equation}\label{<X>}\textstyle
		\<X\>(t) = \sigma^2\int_0^t\d s\int_\T\d y\
		|h_\alpha (u(s\,,y)\,,v(s\,,y))|^2\qquad\forall t\ge0.
	\end{equation}
	Since $\Psi_\alpha^2+\Phi_\alpha^2=1$, we can expand the last square in
	\eqref{h^2} and rewrite $h_\alpha^2$ as
	\[\textstyle
		|h_\alpha(u\,,v)|^2 = v^2 + u^2 - 2uv\Psi_\alpha\left(\frac{u-v}{v}\right)
		= |u-v|^2 +2uv\left[ 1 - \Psi_\alpha\left(\frac{u-v}{v}\right)\right].
	\]
	Because $1-\Psi_\alpha = (1-\Psi_\alpha^2)/(1+\Psi_\alpha)$, the preceding can be rewritten as 
	\begin{align*}
		|h_\alpha(u\,,v)|^2 &= |u-v|^2 +2u\,\frac{\alpha (v-u) \wedge v}{1 
			+ \sqrt{1 - \left(\frac{ \alpha |u-v|}{v}\wedge1\right)}}\\
		&=|u-v|^2 +\frac{ 2\alpha u(v-u) }{1 
			+ \sqrt{1 - \left(\frac{ \alpha |u-v|}{v}\wedge1\right)}}&[\text{see \eqref{D>0}}]\\
		&\ge |u-v|^2 +\alpha u(v-u)\ge \alpha u(v-u)=\alpha uD.
	\end{align*}
	Thus, it follows from \eqref{<X>} that, for all $t>0$,
	\begin{equation}\label{d<X>}\textstyle
		\frac{\d\< X\>(t)}{\d t} \ge\sigma^2 \alpha  \int_\T
		u(t\,,y)D(t\,,y)\,\d y
		\ge \sigma^2 \alpha  \inf_{z\in\T}u(t\,,z) X(t).
	\end{equation}
	According to \eqref{X} and the Dambis and Dubins-Schwarz representation theorem
	(see Revuz and Yor \cite[Chapter V]{RevuzYor}),
	we can find a standard Brownian motion $W=\{W(t)\}_{t\ge0}$ such that,
	for all $t>0$,
	\begin{equation}\label{DDS}\textstyle
		X(t)
		 = X(0) + \sigma \int_0^t\sqrt{ \frac{\d\<X\>(s)}{\d s}}\,\d W(s)
		 \qquad\forall t>0.
	\end{equation}
		Define 
	\[\textstyle
		M(t) = \sigma\int_0^t \sqrt{ \frac{\d\<X\>(s)}{\d s}}\,\d W(s)
		\qquad\forall t>0,
	\]
	and  $M(0)=0$.  Evidently, $M=\{M(t)\}_{t\ge0}$ is a
	continuous local martingale. 
	Since $X\ge0$, we can see that 
	$\d X(t)\le X(t)\,\d t + \d M(t)$ and
	hence by the proof of Proposition A.4 of \cite{KKM2023a},
	applied with $\varepsilon\equiv0$,%
	\footnote{Proposition A.4 of \cite{KKM2023a} is stated 
	under the assumption that $M$ is a continuous $L^2$-martingale
	and that $\d X(t)\le X(t)\,\d t + \d M(t)$, but the proof
	works, without need for any changes, in the present case that $M$ is a continuous
	local martingale since the DDS representation theorem is applicable
	to continuous local martingales owing to localization.}
	\begin{align}\nonumber
		&\P\left\{  \inf_{s\in(0,t)}X(s) > 0 ~,
			\int_0^t \e^{-s}\,\frac{\d\<X\>(s)}{X(s)} \ge b^2\right\}
			\le\sqrt{\frac2\pi}\int_{|x|\le2\sqrt{X(0)}/b} \e^{-x^2/2}\,\d x\\
		&\hskip1.5in\le b^{-1}\sqrt{\tfrac8\pi \|v_0-u_0\|_{L^1(\T)}},\label{B1}
	\end{align}
	uniformly for all $b,t>0$.
	In fact, we do not need the weaker stochastic differential inequality $\d X(t)\le X(t)\,\d t+\d M(t)$.
	Since $\d X(t)=\sigma \d M(t)$, the proof of Proposition A.4 of \cite{KKM2023a} applies
	without need for introducing an exponential term, which accounts for the drift
	$X(t)\,\d t$, and similarly yields the slightly improved bound, 
	\begin{equation}\label{B2}
		\P\left\{  \inf_{s\in(0,t)}X(s) > 0 ~,~
		\int_0^t \frac{\d\<X\>(s)}{X(s)} \ge b^2\right\}
		\le b^{-1}\sqrt{\frac8\pi \|v_0-u_0\|_{L^1(\T)}},
	\end{equation} 	valid uniformly for all $b,t>0$. Either one of the
	inequalities \eqref{B1} or \eqref{B2} are good enough for our needs, since we are interested
	in the regime $t\ll1$.\footnote{In fact, we may deduce the slightly stronger
	inequality \eqref{B2} from \eqref{B1}
	by applying \eqref{B1} with the process $Y$ replaced by 
	$t\mapsto \bar{Y}(t)= \exp(t)Y(t)$. We leave the details
	to the interested reader.}
	But we prefer to use the slightly sharper inequality \eqref{B2}. 
	
	Recall \eqref{D>0}, \eqref{X}, and appeal to \eqref{d<X>}
	to deduce from \eqref{B2}, applied with $b^2
	= \sigma^2\alpha  t \inf_{(0, t]\times  \T} u(s\,, z)$, that
	for all $\eta,t>0$,
	\begin{align}\nonumber
		&\textstyle\P\left\{  \inf_{s\in(0,t)}\|v(s)-u(s)\|_{L^1(\T)} > 0 ~,\,
			 \inf_{(0, t]\times  \T} u(s\,,z)\ge  \frac{1}{2}\,\inf_{z\in \T} u_0(z) \right\}\\
		&\hskip1.5in
			\le \sqrt{\frac{16}
			{\sigma^2\pi \alpha   t }\frac{\|v_0-u_0\|_{L^1(\T)}}{\inf_{z\in \T} u_0(z)} }.
			\label{S1}
	\end{align}
	On the other hand, Lemma \ref{lem:sup:inf} implies that
	there exists a constant $c_*=c_*(f, g)>0$ such that 
	\[\textstyle
		\P\left\{  \inf_{(0, t]\times  \T} u(s\,,z)\le  \frac{1}{2}\,\inf_{z\in \T} u_0(z) \right\}
		\le c_*\exp\left( -c_*^{-1}/\sqrt{t}\right),
	\]
	uniformly for all $t\in(0\,,1)$ and all nonzero and measurable $u_0,v_0:\T\to\R_+$ that satisfy
	\eqref{u_0<v_0}.
	
	Now we may combine the above bound with \eqref{S1} and conclude that
	\begin{align*}
		\P\left\{ \inf_{s\in(0,t)} X(s)>0\right\} 
		&\le \sqrt{\frac{16}{\sigma^2\pi \alpha   t  }
			\frac{\|v_0-u_0\|_{L^1(\T)} }{\inf_{z\in \T} u_0(z)}} + 
			c_*\e^{-c_*^{-1}/\sqrt{t}},
	\end{align*}
	uniformly for all 	nonzero and measurable $u_0,v_0:\T\to\R_+$ that satisfy
	\eqref{u_0<v_0}, and for all $t\in(0\,,1)$. To be sure, we emphasize that the
	implied constant in the above depends only on $f\,, g$. By continuity,
	if $u(s)\neq v(s)$ for some $s\in(0\,,t)$ then $\inf_{s\in(0,t)}X(s)>0$.
	Therefore, we can conclude from \eqref{u_0<v_0} that  we can find 
	$A=A(f\,,g)>1$ such that
	\begin{align*}
		&\P\left\{ u(s)\neq v(s)\text{ for some $s\in(0\,,t]$}\right\}\\
		&\hskip1in\le A\, \left( \frac{\|u_0 - v_0\|_{C(\T)}}{%
			 \alpha   t \inf_{z\in \T} u_0(z) }\right)^{1/2}
			+ A\, \exp\left( -\frac{1}{A\sqrt{t}}\right),
			\end{align*}
	independently of our earlier choice of $\alpha$, and
	uniformly for all $t\in(0\,,1)$ and all bounded, measurable, and non-zero 
	respective initial profiles $u_0,v_0:\T\to\R_+$ that satisfies \eqref{u_0<v_0}. 
	Moreover, when $u_0$ and $v_0$ are not ordered pointwise, we can replace 
	$(u_0\,,v_0\,,u\,,v)$ by $(u_0 \wedge v_0\,,u_0 \vee v_0\,, \underline{u}\,, \overline{u})$,
	as explained earlier. This does not affect $\|u_0 - v_0\|_{C(\T)}$
	or the event $\{u(s) \neq v(s)\text{ for some }s\in(0,t]\}$. 
	In addition, $[\inf_{x\in \T} u_0(x) ]\wedge [ \inf_{x\in \T} v_0(x)] 
	\leq u_0\wedge v_0$.  Hence the same arguments show that the claim 
	holds for all nonzero initial profiles $u_0,v_0$,
	whether or not they are ordered pointwise. This proves the proposition.
\end{proof}

In Proposition \ref{pr:coupling}, we showed that with high probability 
the solutions $u$ and $v$  can be coupled by time $t$. In other words, 
there exists a random time $\tau \leq t$ such that $u(\tau)\equiv v(\tau)$.
The next proposition yields a tail estimate for the difference 
$\|u(s)-v(s)\|_{C(\T)}$ that is valid uniformly over time interval $[0\,, t]$.

\begin{proposition}\label{pr:u-v} 
	There exists $A=A(f\,, g)>0$ such that 
	the following holds for all $\Lambda,K>0$,
	$t \in (0\,, 1)$, and all bounded, measurable, and non-zero  
	initial functions $u_0,v_0:\T\to\R_+$ that satisfy
	$\|u_0-v_0\|_{C(\T)} \le K$:  
	\begin{align*} 
		&\textstyle \P \left\{ \sup_{0\leq s\leq t} \|u(s) - v(s) \|_{C(\T)}
			\ge K + \Lambda   \right\} \leq A\, \exp\left( -A^{-1}/\sqrt{t}\right)  \\
		& \qquad + A\, \exp\left( -\frac{\Lambda^2}{A \left( K+ \Lambda \right)
			\left(K + \Lambda +\alpha \left[  \|u_0\|_{C(\T)}\wedge \|v_0\|_{C(\T)}\right]
			\right)\sqrt{t} }  \right). 
	\end{align*} 
\end{proposition}

\begin{proof}
	We will use the same notation and setup as in the proof 
	of Proposition~\ref{pr:coupling}.   In particular, we assume without loss of generality that 
	$0\le u_{0}  \le v_{0}$. By the comparison theorem (mentioned 
        in the proof of Lemma \ref{lem:dissipation}) the latter implies that
	$u\le v$ almost surely. Define
	\[
		D(t\,,x) = v(t\,,x)- u(t\,,x), \quad D_{0}(x) = D(0\,,x)= v_{0}(x)-u_{0}(x) \ge 0,
	\]
	and let  $\tau_{D} = \inf\{t\ge0:\,\|D(t)\|_{C(\T)} \ge K+\Lambda\}.$
	We aim to estimate
	\[
		\P\{\tau_{D}\le t \} =
		\P\left\{ \sup_{s\in[0,t]}\|D(s)\|_{C(\T)}
		\ge K+\Lambda \right\}.
	\]
	 As in the proof of Proposition \ref{pr:coupling}, we can write
	\[
		D(t\,,x)=(p_{t}*D_{0})(x)+N_D(t\,,x),
	\]
	where
	\[\textstyle
		N_D(t\,,x)=
		\sigma\int_{(0,t)\times\T}  p_{t-s}(y-x)\, 
		h_\alpha \left(u(s\,,y)\,, v(s\,,y)\,, D(s\,,y)\right)\,\zeta(\d s \, \d y).
	\]
	Here $\zeta$ denotes another space--time white noise, and
	\[
		h_\alpha (u\,,v\,,D) = \sqrt{ D^{2}+ \frac{2\alpha u D}{1 +
		\sqrt{1 - \left(\frac{\alpha\,|v-u|}{\,v\,}\right)\wedge1\,}}},
	\] 
	The function $h_\alpha$ is the same function $h_\alpha (u\,,v)$ 
	that appeared earlier in the proof of Proposition~\ref{pr:coupling}; see \eqref{h^2}.
	
	Next we define a truncated version of $N_{D}$.  Let 
	\[\textstyle
		\tilde{h}_\alpha (u\,,v\,,D) = h_\alpha
		\left( u\wedge 2\,\|u_{0}\|_{C(\T)}\,, v\,, D\wedge( K+\Lambda )\right).
	\]
	The function $\tilde{h}$ is bounded by construction.  We may define
	\[\textstyle
		\tilde{N}_D(t\,,x) =\sigma \int_{(0,t)\times\T} 
		p_{t-s}(y-x)\,\tilde{h}_\alpha \left(u(s\,,y)\,,v(s\,,y)\,,D(s\,,y)\right)\,\zeta(\d s,\d y),
	\]
	and observe that $N_D(s\,, x) = \tilde{N}_D(s\,, x)$ for all $s<\tau_D$ a.s.\
	 on the event $\{\tau_{D}\le t ~,~\sup_{0\le s\le t}\|u(s)\|_{C(\T)} \le 2\,\|u_{0}\|_{C(\T)} \}$. 
	 Thus,  we see that a.s.\ on the event 
	$\{\tau_{D}\le t~,~ \sup_{0\le s\le t} \|u(s)\|_{C(\mathbb{T})}\le2\|u_{0}\|_{C(\T)}\}$,  
	\[
		\|D(\tau_{D})\|_{C(\T)} \le \|D_{0}\|_{C(\T)} + 
		\|\tilde{N}_D(\tau_{D})\|_{C(\T)}.
	\]
	However,  $\tau_{D}$ is exactly the first time that
	the left-hand side exceeds $K+\Lambda$ (recall that
	$\|D_0\|_{C(\T)} \le K$).  Thus,  $\|\tilde{N}_D(\tau_{D})\|_{C(\T)} \ge \Lambda,$
	almost surely on the event $\{\tau_{D}\le t~,~\|u(s)\|_{C(\mathbb{T})}\le2\|u_{0}\|_{C(\T)}\}$. 
	It follows that
	\[
		\{\tau_{D}\le t\} \subseteq \left\{\sup_{0\le s\le t}
		\|\tilde{N}_D(s)\|_{C(\T)}\ge\Lambda/2\right\}
		\cup \left\{\sup_{0\le s\le t}\|u(s)\|_{C(\mathbb{T})} >2\|u_{0}\|_{C(\T)}\right\}.
	\]
	We now use Lemma \ref{lem:IX} and Lemma \ref{lem:sup:inf} to 
	bound the probabilities of the events on the right-hand side above. 
	First of all, Lemma \ref{lem:sup:inf} says that there exists a constant 
	$A_1=A_1(f\,, g)>0$ such that 
	\[\textstyle
		\P \left\{ \sup_{0\le s\le t} \|u(s)\|_{C(\T)} 
		\ge2\|u_{0}\|_{C(\T)}\right\} \le A_{1} \exp (-A_1^{-1}/\sqrt{t}).
	\]
	In order to bound the probability that involves $\tilde{N}_D$, we first observe that 
	\[\textstyle
		\left|\tilde{h}_\alpha(u\,,v\,,D)\right|^{2} \le C_{0}
		=\left(K+\Lambda \right) \left( K+\Lambda+4\alpha\,\|u_{0}\|_{C(\T)}\right). 
	\]
	Thus, Lemma \ref{lem:IX} implies that there exists $A_{2} = A_{2}(f\,,g)>0$ such that 
	\[\textstyle
		\P\left\{ \sup_{s\le t}\|\tilde{N}_D(s)\|_{C(\T)}
		\ge \Lambda /2\right\} \leq A_{2}\exp \left( 
		- \Lambda^{2}/[A_{2}\,C_{0}\,\sqrt{t}] \right). 
	\]
	Compile the preceding bounds and relabel the constants to 
complete the proof of Proposition \ref{pr:u-v}.
\end{proof}

\section{Proof of Theorem \ref{th:main}}\label{sec:Pf}
Throughout the proof, let us choose
and fix an arbitrary number
\begin{equation}\label{vareps}
	\varepsilon\in(0\,,\exp(-\e^\e)),
\end{equation}
and define
\begin{equation}\label{delta}
	\delta = \delta(\varepsilon) = [\log\log(1/\varepsilon)]^{-1}.
\end{equation}

Recall \eqref{gamma}, \eqref{tau(T)}, and \eqref{P(tau)},
and observe that
\[
	\P\left\{ \|w(t)\|_{C(\T)} \le \exp(-\gamma t) \quad
	\forall t\ge T_0
	\right\} \ge 1-\varepsilon,
\]
where, for the same non-random constant $L_*=L_*(f,g)>1$ that appeared in \eqref{P(tau)},
\begin{equation}\label{T_0}
	T_0=T_0(f,g\,,\varepsilon) = L_*\log( L_*/\varepsilon)\qquad
	\forall\varepsilon\in(0\,,1).
\end{equation}
We stop to emphasize that $T_0$ is nonrandom. 
With the above in place we define an unbounded sequence
of deterministic numbers $\{T_n\}_{n\in\Z_+}$ as follows: $T_0$ has already been defined. Then, we set
\begin{equation}\label{T_n}\textstyle
	T_{n+1} = T_n +\delta(\log_+n)^{-3}\quad\forall n\in\Z_+.
\end{equation}
The choice of the exponent ``3'' of $\log_+ n$ is not particularly relevant, except that we want
the exponent to be strictly greater than 2.

Of course, $T_1=T_0+\delta$, $T_2=T_0+2\delta$, and
\[\textstyle
	T_{n+1} = T_0 + 2\delta + \delta \sum_{j=2}^n (\log_+ j)^{-3}
	\qquad\text{for all integers $n\ge2$}.
\]
This readily yields that $T_n -T_0 \approx \delta n(\log n)^{-3}$ as $n\to\infty$,
uniformly in $\varepsilon$. In particular,
\begin{equation}\label{T_n:bds}
	T_n -T_0\asymp n\delta (\log_+ n)^{-3}\quad\text{uniformly for all $n\in\N$ and
	$\varepsilon\in(0\,,1)$}.
\end{equation}
The above forms an interval partition of the unbounded time interval
$[T_0\,,\infty)$ as follows:
\begin{equation}\label{I_n}
	I_n = \left[ T_n\,,T_{n+1}\right) \qquad\forall n\in\Z_+.
\end{equation}
The intervals $I_0,I_1,\ldots$ are basically defined by their lengths and the latter are dictated
by a kind of asymptotic regularity of the solution to \eqref{SHE} which, in rough terms, 
ought to ensure that as $n\to\infty$,
\begin{equation}\label{LL0}
	\text{$\log w(t\,,x)\approx \log w(T_n\,,x)$, uniformly for all $t\in I_n$ and $x\in\T$, a.s.}
\end{equation}
And of course the same statement remains valid if we replace $w$ by $u$ everywhere; after all,
\eqref{eq:main} is an example of \eqref{SHE}. Therefore,  as $n\to\infty$, we also have
\begin{equation}\label{LL}
	\text{$\log u(t\,,x)\approx \log u(T_n\,,x)$, uniformly for all $t\in I_n$ and $x\in\T$, a.s.}
\end{equation}
In this way, Theorem \ref{th:main} is reduced to proving that, under the conditions and
notation of Theorem \ref{th:main},  
\[
	\lim_{n\to\infty} \| \log u(T_n) - \log w(T_n) \|_{C(\T)}=0
	\quad\text{with high probability}.
\]
This is more or less the route we take, except we will not prove 
\eqref{LL0} -- hence also \eqref{LL} -- directly. A proper description 
of what we do  requires a little more discussion and can be found in
Item (3), a paragraph above \eqref{A_n} below.

Next, we will use the intervals $\{I_n\}_{n\in\Z_+}$ to define a sequence of parabolic Anderson models,
following the coupling constructions of \S\ref{sec:coupling}.

First, we define a space-time random field $v=\{v(t\,,x)\}_{t\ge T_0,x\in\T}$,
which is defined successively on the intervals $I_0,I_1,I_2,\ldots$ as follows,
$$\left[\begin{split}
	&\partial_t v =\partial_x^2v+  \mu v + \sigma v\dot{W}\qquad\text{on $I_0\times\T$},\\
	&\text{subject to }v(T_0)=w(T_0)\text{ on $\T$}.
\end{split}\right.$$
This defines the process $\{v(t)\}_{t\in[T_0,T_1)}$. Now we proceed inductively
in order to extend the definition of $t\mapsto v(t)$ to all of $[T_0\,,\infty)$:
Having defined $\{v(t)\}_{t\in[T_0,T_n)}$ for some $n\in\Z_+$,
we define the process on the time stage
$I_n$ be defined via
$$\left[\begin{split}
	&\partial_t v =\partial_x^2 v+ \mu v + \sigma v\dot{W}\qquad\text{on $I_n\times\T$},\\
	&\text{subject to }v(T_n)=w(T_n)\text{ on $\T$}.
\end{split}\right.$$

Apply the preceding inductively to complete the description of a random field
$\{v(t\,,x)\}_{t\in[T_0,\infty)\times\T}$. Choose and fix some integer $n\in\Z_+$.
Elementary properties of white noise, and the uniqueness of the solution
to linear SPDEs of the form \eqref{PAM}, together imply that the conditional
law of the process $\{v(t)\}_{t\in I_n}$ given everything by time $T_n$ is the same
as the law of the parabolic Anderson model \eqref{PAM} started at $w(T_n)$
and run for $T_{n+1}-T_n= \delta(\log_+n)^{-3}$ units of time. We pause to
caution however that the process $v$ does not solve a parabolic Anderson model over all time
in $[T_0\,,\infty)$.
In fact, one does not even expect it to be continuous at times $T_1,T_2,\ldots$
since one would informally expect that typically,
\begin{equation}\label{jumps}
	v(T_n) = w(T_n)\neq w(T_{n+1}) = v(T_{n+1}) \qquad\text{a.s.\ $\forall n\in\N$}.
\end{equation} 

The equalities are part of the definition of $v$; the inequality in \eqref{jumps} is the cause of
the potential jump discontinuities. One might in fact expect that the
existence of such discontinuities can be rigorously proved
for many instances of $(f,g)$, using the Malliavin calculus
along the lines of Mueller and Nualart \cite{MN}. 

Recall \eqref{I_n}.
We can think of the random field $v=\{v(t)\}_{t\in[T_0,\infty)}$ as one that is a parabolic Anderson
model in every interval $I_n$ [$n\in\Z_+$] but also tracks the solution $w$ to our SPDE carefully
by making sure that $w$ and $v$ coincide at the left end point $T_n$ of $I_n$ for every $n$. The fact that
$v$ tracks $w$ well in $I_n$ -- and
particularly well when $n\gg1$ -- is a consequence of the fact that the length $|I_n|$ of $I_n$ is
$\delta(\log_+n)^{-3}<1$ for all $n\in\Z_+$, and  especially that $|I_n|\ll1$ when $n\gg1$.

Next, we introduce a coupling of $v$ 
and the solution $u$ to a proper parabolic Anderson model on the time interval $[T_0\,,\infty)$. In order
to do that we will follow the recipe of \S\ref{sec:coupling}. We may enlarge the underlying probability
space, if need be, in order to  
introduce a sequence $\{\dot{W}_n\}_{n\in\Z_+}$ of 
space-time white noises such that $\dot{W},\dot{W}_0,\dot{W}_1,\dot{W}_2,\ldots$
are totally independent.\footnote{Since $\dot{W}$'s are random distributions, we mention
that this independence can be stated rigorously as follows:
The real-valued random variables
$\int\varphi\,\d\dot{W}, \int\varphi_1\,\d\dot{W}_1, \int\varphi_2\,\d\dot{W}_2,\ldots$ -- all defined
as Wiener integrals --
are totally independent for all possible choices of $\varphi,\varphi_n\in L^2(\R_+\times\T)$
$[n\in\N]$.
}
Recall the functions $\Phi$ and $\Psi$ from \eqref{Phi:W};
more precisely, we apply \eqref{Phi:W} with $\alpha=1$  and let $\Phi=\Phi_1$ and $\Psi=\Psi_1$. 
We can now follow \eqref{PAM:v} and define a random field
$\{u(t)\}_{t\in[T_0,T_1]}$ as follows:
$$\left[\begin{split}
	&\textstyle\partial_t u = \partial^2_x u + \mu u + \sigma u\left\{ \Psi\left(\frac{v-u}{u}\right) \dot{W}+
        		\Phi\left(\frac{v-u}{u}\right)\dot{W}_0\right\}\  \text{on $I_0\times\T$}, \\
	&\text{subject to }u(T_0)=v(T_0)\text{ on $\T$}.
\end{split}\right.$$
Because $u$ and $v$ agree at their initial time $T_0$, it follows from the coupling that
$u=v$ on the entire interval $[T_0\,,T_1)=I_0$. However, the preceding
construction will illuminate the remainder of the construction of $u$, which ends up
being different from $v$ in subsequent intervals $I_1,I_2,\ldots .$ From here on 
out we proceed inductively.

Let us suppose that $\{u(t)\}_{t\in I_{n-1}}$ has been constructed for some $n\in\N$,
and let us think of this as the $(n-1)$st stage of the construction of the random field $u$.
We can then extend the definition of $u$ to the time interval $[T_0,T_{n+1})$
by defining the next ($n$th) stage $\{u(t)\}_{t\in I_n}$ of the construction
as the solution to the following:
$$\left[\begin{split}
	&\textstyle\partial_t u = \partial^2_x u + \mu u + \sigma u\left\{ \Psi_n \left(\frac{v-u}{u}\right) \dot{W}+
        		\Phi_n \left(\frac{v-u}{u}\right)\dot{W}_n \right\}\ 
			\text{on $I_n\times\T$}, \\
	&\text{subject to }u(T_n)\text{ being defined in the previous stage everywhere on $\T$},
\end{split}\right.$$
where, for all $y\in\R$,
\begin{equation}\label{Phi:W:n}
\begin{aligned} 
	&\textstyle\Phi_n(y)=\Phi_{\alpha_n}(y) = \sqrt{\alpha_n|y|\wedge 1},\\
	&\textstyle\Psi_n(y)=\Psi_{\alpha_n}(y) = \sqrt{1-|\Phi_n(y)|^2} = \sqrt{1-(\alpha_n|y|\wedge 1)},
\end{aligned}
\end{equation} and where $\alpha_n$ is chosen later.

We have reversed the roles of the notations for $u$ and $v$ in \eqref{PAM:v}, but are 
otherwise following that construction verbatim. The preceding inductively constructs a 
space-time random field $\{u(t\,,x)\}_{(t,x)\in[T_0,\infty)\times\T}$. Because
$u$ is continuous,  regularity results about SPDEs, and their Markov properties, together ensure that
$u$ solves a parabolic Anderson model whose law is the same as that in \eqref{PAM},
except the process starts at time $T_0$ (not $0$) and the initial value is $v(T_0) = w(T_0)$,
which is random but in a non-anticipatory manner.
The random field $\{u(t)\}_{t\in [T_0,\infty)}$ is the one that is announced in Theorem \ref{th:main}.
From here on, the strategy of the proof will be to quantify the following intuitive properties:
\begin{enumerate}
\item Since $v(T_n)=w(T_n)$ and $I_n$ is a small interval for all $n\geq 1$ 
	when $\varepsilon \ll1$, $v\approx w$
	on $I_n$ with high probability for all $n\geq 1$;
\item The coupling of $(v\,,u)$ is successful in $I_n$ with high probability for all
	$n\geq 1$ when $\varepsilon \ll 1$.
	Recall that this means that, with high probability, $u(t)=v(t)$ for some $t\in I_n$.
	By the Markov property, once $u$ and $v$ coincide, then they are equal from that time on. 
	This and the previous step together would imply that with high probability there exists
	$t\in I_n$ such that $u(s)=v(s)$ for all $s\in[t\,,T_{n+1}]$ and in particular
	$u(T_{n+1})=v(T_{n+1}-)$.  Therefore
	if we can show this holds for all $n\geq 1$
	then $w(T_n)$ and $u(T_n)$ are close with high probability for all  $n\ge1$.
\item Finally, we combine the first two steps to show that $u(t)$ is close to 
	$w(t)$ on $I_n$ for all $n\geq 1$. More precisely, from Step (2), 
we can have that, at the discrete times $T_n$,   $u(T_n)$ and $w(T_n)$ are 
close for all $n\geq 1$ with high probability. Since $v(T_n) = w(T_n)$ by 
construction, it follows that $u(T_n)$ and $v(T_n)$ are close for all 
$n\geq 1$. Then, by Proposition \ref{pr:u-v}, we conclude that $u$ and $v$ 
remain close throughout the interval $I_n$. Combining this with Step (1), 
where $v$ is shown to be close to $w$ over the entire interval $I_n$, we 
deduce that $u$ must also remain close to $w$ in each $I_n$ with high 
probability  for all $n \ge 1$. A suitable formulation of this property 
yields \eqref{LL0} and \eqref{LL}. 
\end{enumerate}

We now introduce a sequence of events $\bm{A}_0,\bm{A}_1,\ldots$
that depend on the construction above, as well as a sequence of positive numbers
$\varepsilon_1>\varepsilon_2>\cdots>0$ whose numerical
values can be found in \eqref{epsilon_n} below.
It is perhaps worth mentioning that $\{\varepsilon_n\}_{n=1}^\infty$ is a sequence of 
numbers that depends on the number $\varepsilon\in(0\,,\e^{-\e})$ that has been 
already chosen and fixed  in \eqref{vareps}. With the sequence $\{\varepsilon_n\}_{n=1}^\infty$
in place, we define
\begin{align}
	\bm{A}_0 &= \left\{ \omega\in\Omega:\ \tau(T_0)(\omega)=\infty  \right\},
		\label{A_n}\\\nonumber
	\bm{A}_n &= \left\{ \omega\in\Omega:\ \| w(T_n)-u(T_n)\|_{C(\T)}(\omega) \le
		\varepsilon_n \| w(T_n)\|_{C(\T)}(\omega) \right\}\quad
		\forall n\in\N.
\end{align}
The strategy of our proof of Theorem \ref{th:main}
is to establish a lower bound for probabilities of the form,
\begin{equation}\label{P_N}
	\mathcal{P}_N = \P(  \bm{A}_1 \cap \cdots\cap
	\bm{A}_N \mid \bm{A}_0) 
	\qquad\forall N\in\N,
\end{equation}
in order to show that the above can be made to be close to one when $N\gg1$
and $\varepsilon\ll1$,
as long as the various parameters of the construction are selected judiciously.
Thanks to \eqref{T_0} and \eqref{P(tau)},
\begin{equation}\label{P(A_0)}
	\P(\bm{A}_0) \ge 1 - \varepsilon,
\end{equation}
where $\varepsilon$ was fixed in \eqref{vareps}.
This shows that the conditioning in \eqref{P_N} is ``soft''
in the sense that we are conditioning on a set of very large measure.

In order to estimate $\mathcal{P}_N$ from below we write
\[\textstyle
	1-\mathcal{P}_N
	= \P\big( \bm{A}_1^\mathsf{c}\mid\bm{A}_0\big)+
	 \sum_{k=2}^N \P\big( \bm{A}_1 \cap\cdots\cap\bm{A}_{k-1}\cap
	\bm{A}_k^\mathsf{c}  \mid \bm{A}_0\big),
\]
which is another way to state that, in order for the event
$\bm{A}_1\cap\cdots\cap \bm{A}_N$ to fail to hold,
there must be a first index $k\in\{1,\ldots,N\}$ such that
$\bm{A}_k$ fails to hold.
Consequently, \eqref{P(A_0)} yields the following bound, which in turn will determine the
strategy of the remaining portion of the proof of Theorem \ref{th:main}:
\begin{equation}\label{aim}\textstyle
	1-\mathcal{P}_N
	\le (1-\varepsilon)^{-1}\left[\P\big( \bm{A}_0\cap \bm{A}_1^\mathsf{c} \big)+
	 \sum_{k=2}^N \P\big( \cap_{j=0}^{k-1}\bm{A}_j\cap
	\bm{A}_k^\mathsf{c}\big)\right].
\end{equation}
The quantity inside the square brackets is a sum of $N$ terms.
We will estimate the first two terms first -- these are stages 1 and 2 of the argument
-- and then construct
an inductive argument for estimating the remainder from the first two.

\subsection{The first stage}
The task of this first stage of the proof is to
bound the first probability term $\P( \bm{A}_0\cap \bm{A}_1^\mathsf{c} ) $
that appears on the right-hand side of \eqref{aim}. For the remainder of 
this paper, we will use successive conditioning
and the Markov property for the particular coupling construction of this paper. Therefore,
we enhance the original filtration $\mathscr{F}$ of the noise by adding to $\mathscr{F}(t)$
the sigma-algebras generated by $\int_{[0,s]\times\T]}\varphi(y)\,W_j(\d r\,\d y)$
as $(s\,,\varphi\,,j)$ roams over $[0\,,t]\times L^2(\T)\times\Z_+$. We continue to denote
the resulting enlarged filtration by $\mathscr{F}(t)$ as this should not cause 
confusion thanks to the independence of $W,W_0,W_1,\ldots$.

We may appeal to Proposition \ref{pr:tracking} to see that for every $k>6$ there exists $L=L(k\,,f,g)>0$,
and independent of the value of $\varepsilon$, such
that
\begin{equation}\label{Mk:T0}\textstyle
	\E_{\mathscr{F}(T_0)}\left( \| w_{T_0} - u_{T_0}\|_{C(I_0\times\T)}^k \right)
	\le L \left[\mathscr{E}\left(\e^{-\gamma T_0}\right)\right]^k  \|w(T_0)\|_{C(\T)}^k,
\end{equation}
a.s.\ on $\{\tau(T_0)>T_0\}$. It might help to recall that
$\mathscr{E}$ was defined in \eqref{R:star}. 
Thus, we have that  a.s.\ on $\{\tau(T_0)>T_0\}$
\begin{align}\nonumber
	&\textstyle\P_{\mathscr{F}(T_0)} \left\{ \| w(T_1) - u(T_1) \|_{C(\T)} \ge  
		\varepsilon_1 \|w(T_1)\|_{C(\T)}  \text{ and }
		\tau(T_0) =\infty\right\}\\\nonumber
	&\textstyle\le \P_{\mathscr{F}(T_0)}\left\{ \| w_{T_0}(T_1) - u_{T_0}(T_1) \|_{C(\T)} 
		\ge \varepsilon_1\,	\inf_{T_0 \leq s \leq T_1} \inf_{z\in \T} w(s\,, z) \right\} \\
	&\textstyle\le \P_{\mathscr{F}(T_0)}\left\{ \| w_{T_0}(T_1) - u_{T_0}(T_1) \|_{C(\T)} 
		\ge \frac{\varepsilon_1}{2} \inf_{z\in \T} w(T_0\,, z) \right\} \label{Stage1}\\ \nonumber
	&\textstyle\qquad \quad  + \P_{\mathscr{F}(T_0)}\left\{    
		\inf_{T_0\leq s\leq T_1} \inf_{z\in \T} w(s\,, z) \leq 
		\frac12 \inf_{z\in \T} w(T_0\,, z) \right\} \nonumber \\\nonumber
	&\textstyle\le L \left( \frac{2}{\varepsilon_1} \mathscr{E}\left(\e^{-\gamma T_0}  \right) 
		\|w(T_0)\|_{C(\T)} /  \inf_{z\in \T} w(T_0\,, z) \right)^k
		+ A\e^{- A^{-1}/\sqrt{T_1 - T_0}},
\end{align}
where we have used \eqref{Mk:T0} (yielding the constant $L=L(k\,,f, g)$) 
and Lemma~\ref{lem:sup:inf} (yielding the constant $A=A(f,g)$).

In order to control the ratio of the supremum and infimum, 
we  define an event $\tilde{\bm{A}}_0$ as 
\[ 
	\tilde{\bm{A}}_0 : = \left\{ \omega \in \Omega  :\ 
	\| w(T_0)\|_{C(\T)}    \leq \eta_0^{-1} \inf_{z\in \T} w(T_0\,, z) \right\},
\]
where $\eta_0=|\log\delta|^{-1}$.  We now obtain that for every $k>6$ and $\ell>2$, 
\begin{equation}
\begin{split}
	\P&\big( \bm{A}_0\cap\bm{A}_1^\mathsf{c}\big) 
		= \E\left[\P_{\mathscr{F}(T_0)} \big( \bm{A}_1^\mathsf{c}\big)
		;\ \bm{A}_0\right] \\
	&\leq \E\left[\P_{\mathscr{F}(T_0)} \big( \bm{A}_1^\mathsf{c}\big)
		;\ \bm{A}_0\cap \tilde{\bm{A_0}}\right]  + 
		\P\big( \tilde{\bm{A}}_0^\mathsf{c} \big) 
		\label{P(w(T_1)}\\ 
	&\le L \left[ \varepsilon_1^{-1}
		\eta_0^{-1} 2\mathscr{E}\left(\e^{-\gamma T_0}\right)
		\right]^k +  A\exp\left( - A^{-1}/\sqrt{T_1 - T_0} \right) + K \eta_0^\ell, 
\end{split}
\end{equation}
 where we have used \eqref{Stage1} and  Proposition \ref{pr:oscillation} (yielding the universal constant $K>0$).  This concludes our estimate for the first term on the right-hand side of \eqref{aim}.

\subsection{The second stage}
In the second stage of the argument we complement the estimate \eqref{P(w(T_1)}
by bounding the second probability 
$\P( \bm{A}_0\cap\bm{A}_1\cap\bm{A}_2^\mathsf{c})$ that appears on the
right-hand side of \eqref{aim}. That estimate is done in a similar manner to the one in the previous
stage, except it is complicated by the fact that $u\not\equiv v$ on $I_1$, in contrast 
to the fact that $u=v$ on $I_0$ which simplified the analysis in the first stage. This complication
manifests itself right away, as the analogue of \eqref{Mk:T0} is only the following, 
because $w(T_1)$ is equal to $v(T_1)$ and not $u(T_1)$:
By Proposition \ref{pr:tracking}, for every $k>6$ there exists $L=L(k\,,f,g)>0$,
and independent of the value of $\varepsilon$, such
that 
\[
	\E_{\mathscr{F}(T_1)}\left( \| w_{T_1} - v_{T_1}\|_{C(I_1\times\T)}^k \right)
	\le L \left[\mathscr{E}\left(\e^{-\gamma T_1}\right)\right]^k  \|w(T_1)\|_{C(\T)}^k,
\]
a.s.\ on $\{\tau(T_1)>T_1\}$. Therefore, Chebyshev's inequality and 
Lemma \ref{lem:sup:inf} yield the following: There exists a constant $L=L(f, g)>0$ and $A=A(f, g)>0$ such that   
\begin{align}\nonumber
	&\P_{\mathscr{F}(T_1)}\left\{ \| w(T_2) - v(T_2-) \|_{C(\T)}
		\ge  \varepsilon_2\|w(T_2)\|_{C(\T)}
		\text{ and }\tau(T_1)=\infty\right\}\\\nonumber
	&\le \P_{\mathscr{F}(T_1)}\left\{ \| w_{T_1} - v_{T_1}\|_{C(I_1\times\T)}
		\ge \varepsilon_2\|w(T_2)\|_{C(\T)} \right\}
		 \\
	&\textstyle\le\P_{\mathscr{F}(T_1)}\left\{ \| w_{T_1} - v_{T_1}\|_{C(I_1\times\T)}
		\ge \varepsilon_2  \inf_{T_1\leq s\leq T_2} \inf_{z\in \T} w(s\,, z) \right\}
		 \label{pre-Stage2}\\\nonumber
	&\textstyle\le\P_{\mathscr{F}(T_1)}\left\{ \| w_{T_1} - v_{T_1}\|_{C(I_1\times\T)}
		\ge \frac{\varepsilon_2}{2}  \inf_{z\in \T} w(T_1\,, z) \right\}\\\nonumber
	&\textstyle \qquad \quad  + \P_{\mathscr{F}(T_1)}\left\{
		\inf_{T_1\leq s\leq T_2} \inf_{z\in \T} w(s\,, z) \leq 
		\frac12 \inf_{z\in \T} w(T_1\,, z)\right\}\\\nonumber
	&\le L \left( \frac{2}{\varepsilon_2}
		\mathscr{E}(\e^{-\gamma T_1}) \frac{\|w(T_1)\|_{C(\T)}}{%
		\inf_{z\in \T} w(T_1\,, z)} \right)^k
		+ A\exp\left( -\frac{1}{A\sqrt{T_2 - T_1}} \right),
\end{align} 
a.s.\ on $\{\tau(T_1)>T_1\}$.  The preceding is an appropriate 
version of a conditional analogue of \eqref{Stage1}.

Note that we may not write $u_{T_1}$ in place of
$v_{T_1}$ now, since $u(T_1)$ need not be equal to $w(T_1)$ [see the requirements of
Proposition \ref{pr:tracking}]. We can now address this using the coupling estimate 
in Proposition \ref{pr:coupling:1}. Indeed, that result and the Markov property
together imply that  we can find a non-random number 
$A=A(f\,,g)>1$ such that, almost surely, 
\begin{align*}
	&\P_{\mathscr{F}(T_1)}\left\{ u \neq v\text{ everywhere in $I_1$}\right\}\\
	&\le A\left( \frac{\|u(T_1)-v(T_1)\|_{C(\T)}}{\alpha_1 (T_2-T_1)  \inf_{z\in \T} v(T_1\,,z)}
		\right)^{1/2} + A\exp\left( - \frac{1}{A\sqrt{T_2-T_1}}\right)\\
	&= A\left( \frac{\|u(T_1)-w(T_1)\|_{C(\T)}}{\alpha_1 (T_2-T_1) \inf_{z\in \T} w(T_1\,, z) }
		\right)^{1/2} + A\exp\left( - \frac{1}{A\sqrt{T_2-T_1}}\right).
\end{align*}
Therefore, almost surely on $\bm{A}_1$,
\begin{align*}
	&\P_{\mathscr{F}(T_1)}\left\{ u \neq v\text{ everywhere in $I_1$}\right\} \\
	&\quad 	\le A\left( \frac{\varepsilon_1 \|w(T_1)\|_{C(\T)}  }{
		\alpha_1( T_2-T_1 ) \inf_{z\in \T} w(T_1\,, z)}\right)^{1/2} 
		+ A\exp\left( - \frac{1}{A\sqrt{T_2-T_1}}\right).
\end{align*}
Recall that once $u(t)=v(t)$ for some time $t\in I_1$, 
then $u(s)=v(s)$ for all $s\in[t\,,\infty)\cap I_1 = [t\,,T_2)$,
and in particular $u(T_2)=v(T_2-)$.
Consequently, almost surely on $\bm{A}_1$, 
\begin{align}\label{pre:Stage2:1}
	& \P_{\mathscr{F}(T_1)}\left\{ u(T_2) \neq v(T_2-)\right\}   \\\nonumber
	&\quad 	\le A\left( \frac{\varepsilon_1 \|w(T_1)\|_{C(\T)}  }{
		\alpha_1(T_2-T_1) \inf_{z\in \T} w(T_1\,, z)}\right)^{1/2}
		+ A\exp\left( -\frac{1}{A\sqrt{T_2-T_1}}\right).
\end{align} 
Because $\tau(T_1)=\infty$ a.s.\ on $\bm{A}_0$, 
we can deduce from \eqref{pre-Stage2} that,
almost surely on $\bm{A}_0\cap\bm{A}_1$,
\begin{align} \nonumber
	&\P_{\mathscr{F}(T_1)} \left( \bm{A}_2^\mathsf{c}\cap\{\tau(T_1)=\infty\}\right)\\\nonumber
	&=\P_{\mathscr{F}(T_1)}\left\{ \| w(T_2) - u(T_2) \|_{C(\T)}
		\ge \varepsilon_2\|w(T_2)\|_{C(\T)} 
		\text{ and }\tau(T_1)=\infty\right\}\\\nonumber
	&\leq \P_{\mathscr{F}(T_1)}\left\{ \| w_{T_1}(T_2) - v_{T_1} (T_2-) \|_{C(\T)}
		\ge \varepsilon_2\|w(T_2)\|_{C(\T)} \right\} \\\nonumber
        &\qquad\qquad +  \P_{\mathscr{F}(T_1)}
		\left\{ u(T_2) \neq v(T_2-)\right\}  \\\nonumber
	&\le L \left(\frac{2\mathscr{E}\left(\e^{-\gamma T_1}\right)
		\|w(T_1)\|_{C(\T)} }{\varepsilon_2 \inf_{z\in \T} w(T_1\,, z)  }\right)^k
		+A\left( \frac{\varepsilon_1 \|w(T_1)\|_{C(\T)}  }{\alpha_1(T_2-T_1)
		\inf_{z\in \T} w(T_1\,, z)}\right)^{1/2} \\
	&\qquad \qquad + 2A\exp\left( - \frac{1}{A\sqrt{T_2-T_1}}\right).\label{eq:A2_c}
\end{align} 
As in the first stage, we control the ratio of the supremum to the infimum 
by defining the event
\[\textstyle
	\tilde{\bm{A}}_1   =  \left\{ \omega \in \Omega : 
	\| w(T_1)\|_{C(\T)} \le\eta_1^{-1} \inf_{z\in \T} w(T_1, z) \right\},
\]
where $\eta_1=\eta_0=1/\log(1/\delta)$.  We then proceed exactly as in the first stage. That is, we now take expectations
of the preceding to find the following bounds for the unconditional probability of Stage 2: For every $k>6$ and $\ell>2$, 
\begin{align}\nonumber
	&\P\big( \bm{A}_0\cap \bm{A}_1\cap\bm{A}_2^\mathsf{c}\big)
		=\E\left[\P_{\mathscr{F}(T_1)} \left( \bm{A}_2^\mathsf{c}\cap\{\tau(T_1)=\infty\}\right)
		;\ \bm{A}_0\cap\bm{A}_1\right]   \\
	&\leq \E 	\left[\P_{\mathscr{F}(T_1)} \left( \bm{A}_2^\mathsf{c}\cap\{\tau(T_1)=\infty\}\right)
		;\ \bm{A}_0\cap\bm{A}_1\cap \tilde{\bm{A}_1}\right]
		+ \P\big( \tilde{\bm{A}_1}^\mathsf{c} \big) \label{Stage2} \\\nonumber
	&\le L \left(\frac{2\mathscr{E}\left(\e^{-\gamma T_1}\right)  }{%
		\varepsilon_2 \eta_1 }\right)^k  +A\left( \frac{\varepsilon_1 }{%
		\alpha_1 \eta_1 (T_2-T_1)}\right)^{1/2} +{ 2A
		\e^{-A^{-1}/\sqrt{T_2 - T_1}} + K \eta_1^\ell,}
\end{align}
where the last inequality comes from  \eqref{eq:A2_c} and Proposition \ref{pr:oscillation} 
(that yields the universal constant $K>0$ above). This concludes Stage 2.

\subsection{The inductive stage}
The probability estimates for stages $n\ge 3$ are derived inductively, 
each step being similar to the previous stage's estimates
and follow essentially exactly the same route. Therefore we outline the arguments only.
Choose and fix an arbitrary integer $n\ge 2$. The same argument that led to \eqref{pre-Stage2} 
yields the following: For every $k>6$ there exists $L=L(k,f,g)>0$ such that 
\begin{align*}
	&\P_{\mathscr{F}(T_n)}\left\{ \| w(T_{n+1}) - v(T_{n+1}-) \|_{C(\T)}
		\ge  \varepsilon_{n+1}\|w(T_{n+1})\|_{C(\T)}
		\text{ and }\tau(T_n)=\infty\right\}\\\nonumber
	&\le L \left(\frac{2\mathscr{E}\left(\e^{-\gamma T_n}\right)
		\|w(T_n)\|_{C(\T)} }{\varepsilon_{n+1} \inf_{z\in \T} w(T_n\,, z)  }\right)^k
		+ A\exp\left( -\frac{1}{A\sqrt{T_{n+1} - T_n}} \right) ,\label{pre-Stage2}
\end{align*} 
a.s.\ on $\{\tau(T_n)>T_n\}$. We pause to emphasize that,
among other things, $L$ and $A$ do  not depend on the index $n$ as they are 
the same constant that appeared in Proposition \ref{pr:tracking} and Lemma \ref{lem:sup:inf}.

Next, we follow the derivation of \eqref{pre:Stage2:1} and adapt it to the 
present setting in order to find that  we can find a non-random number 
$A=A(f\,,g)>1$ such that
\begin{align*}
	&\P_{\mathscr{F}(T_n)}\left\{ u(T_{n+1}) \neq v(T_{n+1}-)\right\} \\
	&\quad 	\le A\left( \frac{\varepsilon_n \|w(T_n)\|_{C(\T)}  }{%
		\alpha_n(T_{n+1}-T_n) \inf_{z\in \T} w(T_n\,, z)}\right)^{1/2} 
		+ A\exp\left( -\frac{1}{A\sqrt{T_{n+1}-T_n}}\right).
\end{align*}
valid almost surely on $\bm{A}_0\cap\bm{A}_{n-1}$. Once again, the constant
$A$ does not depend on the index $n$. 

As in the earlier stages, we must also control 
$\| w(T_n)\|_{C(\T)}/ \inf_{z\in\T} w(T_n,z)$.  To do this, for any fixed constant $\eta>0$, we define 
\begin{equation}\label{eta_n}
	\eta_n= n^{-\eta}|\log \delta|^{-1}  \ \ \forall n\in\N,
\end{equation}
and 
\[\textstyle
	\tilde{\bm{A}}_n \;=\; 
	\left\{ \omega \in \Omega :
	\| w(T_n)\|_{C(\T)} (\omega)\le \eta_n^{-1} \inf_{z\in \T} w(T_n,z) (\omega)
	\right\}. 
\]
  We now combine the preceding two estimates in the same manner as was done in
\eqref{Stage2} in order to see that for every $k> 6$ and $\ell>2$,  
\begin{equation*}
\begin{split}
	&\P\big( \bm{A}_0\cap\bm{A}_1\cap\cdots\cap\bm{A}_n\cap\bm{A}_{n+1}^\mathsf{c}\big)
		\le \P\big( \bm{A}_0\cap\bm{A}_n\cap\tilde{\bm{A}}_n\cap
		\bm{A}_{n+1}^\mathsf{c}\big) + \P\big( \tilde{\bm{A}}_n^\mathsf{c} \big)  \\
		\nonumber
	&\le L \left(\frac{2\mathscr{E}\left(\e^{-\gamma T_n}\right)}{
		\varepsilon_{n+1} \, \eta_n }\right)^k  +
		A\left( \frac{\varepsilon_n }{\alpha_n \, \eta_n\,  (T_{n+1}-T_n)}\right)^{1/2}\\
		& \qquad +2A\e^{-1/[A\sqrt{T_{n+1} - T_n}]} + K \eta_n^\ell. 
\end{split}
\end{equation*}
Since the first term, which has the form $L(\,\cdots)^k$, is a probability upper bound it could
be replaced with the minimum of the same quantity and one at no cost. This yields the 
following bound:
\[\textstyle
	\P\big( \bm{A}_0\cap\bm{A}_1\cap\cdots\cap\bm{A}_n\cap\bm{A}_{n+1}^\mathsf{c}\big)
	\lesssim \sum_{i=1}^4 Q_i(n),
\]
where
\begin{align*}
	& Q_1(n)= 1\wedge \left(\frac{%
		\mathscr{E}\left(\e^{-\gamma T_n}\right)}{\varepsilon_{n+1}\, \eta_n}\right)^k,
	&& Q_2(n)=\left( \frac{\varepsilon_n}{\alpha_n \, \eta_n\, (T_{n+1}-T_{n})}\right)^{1/2},\\
	& Q_3(n)=\e^{-B/\sqrt{T_n-T_{n-1}}},
        && Q_4(n)= \eta_n^\ell,
\end{align*}
and $B=1/A\in(0\,,1)$ depends only on $(f,g)$ and the implied constant depends only 
on $(k,f,g)$. We pause to emphasize that neither the implied constants, nor the
constants $\gamma, \ell , B$, depend on the particular choices
for the sequences $\{T_n\}_{n=0}^\infty\,, \{\varepsilon_n\}_{n=0}^\infty\,, \{\eta_n\}_{n=0}^\infty$, and  $\{\alpha_n\}_{n=0}^\infty$.
In particular, those constants do not depend on $\varepsilon$; see also 
\eqref{T_0} and \eqref{delta}.

We plan to estimate $\sum_{n=1}^\infty Q_i(n)$, in this order, for $i=4, 3, 2, 1$
saving the more computationally intensive bounds for last. 

The easiest quantity to estimate is $\sum_{n=0}^\infty Q_4(n)$. 
Given $\eta>0$, we choose  $\ell \ge 2$ such that $\eta\ell > 1$. Thus, we see that
\begin{align}\nonumber\textstyle
	\sum_{n=0}^\infty Q_4(n)&\textstyle
		= \sum_{n=0}^\infty \eta_n^\ell = \left|  \log(1/\delta)
		\right|^{-\ell}\left( 1+  \sum_{n=1}^\infty n^{-\eta \ell} \right)\\
	&\lesssim \left[\log\log\log(1/\varepsilon)\right]^{-\ell};\label{T4}
\end{align}
see \eqref{delta}. We observe that the implied constants do not depend on the
parameter $\varepsilon\in(0\,,\exp(-\e^\e))$, and here $\ell$ 
is fixed once $\eta>0$ is chosen, although we will only specify $\eta$ 
later when estimating $\sum_{n=1}^\infty Q_1(n)$.

Next we notice that $\sum_{n=1}^\infty Q_3(n)$ is equal to
\begin{equation}\label{T3}\textstyle
	\sum_{n=1}^\infty\exp\left( - B(\log_+ n)^{3/2}/\sqrt\delta \right)
	\lesssim \delta^{1/3}=\left( \log|\log\varepsilon|\right)^{-1/3},
\end{equation}
uniformly for all $\varepsilon\in(0\,,\exp(-\e^\e))$.
See Lemma \ref{lem:E} of the appendix for an explanation of
the inequality ``$\lesssim\delta^{1/3}$,'' see \eqref{vareps} for 
the choice of the fixed constant $\varepsilon$, and \eqref{delta} for the 
final identity.

So far, the particular constructions of the sequences 
$\{\alpha_n\}_{n=1}^\infty$  and $\{\varepsilon_n\}_{n=1}^\infty$ have not been relevant. In order to properly estimate $\sum_{n=1}^\infty Q_2(n)$
we now make explicit our sequences $\{\alpha_n\}_{n=1}^\infty$  and 
$\{\varepsilon_n\}_{n=1}^\infty$ as follows: For the same
fixed constant $\varepsilon$ as in \eqref{vareps},
\begin{equation}\label{alpha_n}
	\alpha_n = \eta_n =
	n^{-\eta}|\log\delta|^{-1} = n^{-\eta}|\log\log\log(1/\varepsilon)|^{-1}, 
\end{equation}and 
\begin{equation}\label{epsilon_n}
	\varepsilon_n = \frac{\delta n^{-2-4\eta}}{|\log\delta|^4}
	= \frac{n^{-2-4\eta}}{\log\log(1/\varepsilon) \cdot
	|\log\log\log(1/\varepsilon)|^4};
\end{equation} 
see \eqref{delta} and \eqref{eta_n}. For this particular construction of $\{\varepsilon_n\}_{n=1}^\infty$ and $\{\alpha_n\}_{n=1}^\infty$,
we have [by \eqref{T_n}],
\begin{equation}\label{T2}\begin{split}\textstyle
	 \sum_{n=1}^\infty Q_2(n)
	 	&\textstyle\lesssim | \log\delta|^{-1}\sum_{n=1}^\infty n^{-(1+\eta)}
		(\log_+ n)^{3/2} \lesssim  |\log\delta|^{-1}\\
	& = \left[\log\log\log(1/\varepsilon)\right]^{-1}. 
\end{split}\end{equation}
Note that the implied constant
does not depend on $\varepsilon\in\left(0\,,\exp\left(-\e^\e\right)\right)$.

Finally, we analyze $\sum_{n=1}^\infty Q_1(n)$ using the same
sequences $\{\varepsilon_n\}_{n=1}^\infty$ and $\{\eta_n\}_{n=1}^\infty$  that have been defined,
and fixed, in \eqref{epsilon_n} and \eqref{eta_n}:
\[\textstyle
	\sum_{n=1}^\infty Q_1(n)
	\lesssim\sum_{n=1}^\infty \left[1\wedge \delta^{-1}
	|\log\delta|^5 \, n^{2+5\eta}\, \mathscr{E}\left(\e^{-\gamma T_n}\right) \right]^k,
\]
where the implied constant does not depend on the parameter $\varepsilon$
that was fixed in \eqref{vareps}. The estimation of
$\sum_{n=1}^\infty Q_1(n)$ is straightforward but requires a little bit of work, therefore we do the work
slowly to identify parameter dependencies of interest.

Thanks to \eqref{T_0} and \eqref{T_n:bds}, for every $\rho\in(0\,,1)$
we can find a real number $c>0$ such that
\[
	\e^{-\gamma T_n} \le \exp\left( -\gamma L_*\log( L_*/\varepsilon)
	- c\gamma \delta n (\log_+ n)^{-3}\right)\lesssim \varepsilon^{\gamma L_*}
	\wedge \exp\left( - c\gamma \delta n^\rho\right),
\]
where we note that $c$ depends only on $\rho$ and nothing else. The monotonicity of
the mapping $\mathscr{E}$ -- see \eqref{R:star} -- yields
\[
	\mathscr{E}\left( \e^{-\gamma T_n}\right) \le 
	\mathscr{E}\left( c_1\varepsilon^{\gamma L_*}\right)\wedge
	\mathscr{E}\big( c_1\e^{-c_2 \delta n^\rho }\big)\lesssim
	|\log\varepsilon|^{-\chi}\wedge (\delta n^\rho)^{-\chi},
\]
where $\chi$ comes from part (3) of Assumption \ref{ass:par}. In this way we find that
\begin{align*}\textstyle
	\sum_{n=1}^\infty Q_1(n)
		&\textstyle\lesssim \sum_{n=1}^\infty \left[1\wedge \delta^{-1}
			|\log\delta|^5 \, n^{2+5\eta}\,\left\{|\log\varepsilon|^{-\chi}
			\wedge (\delta n^\rho)^{-\chi}\right\} \right]^k\\
&=R_1+R_2,
\end{align*}
where 
\[\textstyle
	R_1=\sum_{n\le \left(|\log\varepsilon|/\delta\right)^{1/\rho}}
	\left( \delta^{-1} |\log\delta|^5 |\log\varepsilon|^{-\chi}\, 
	n^{2+5\eta }\right)^k 
\]
and
\[\textstyle
	R_2= \sum_{n>\left(|\log\varepsilon|/\delta\right)^{1/\rho}}
	\left( \delta^{-(1+\chi)} |\log\delta|^5 \,
	n^{-\rho\chi + 2+5\eta} \right)^k.
\]
We first consider $R_1$: 	
\begin{equation}\label{R_1}
\begin{split}
	R_1 &\textstyle\le \left( \delta^{-1}|\log\delta|^5 /|\log\varepsilon|^\chi
		\right)^k \sum_{n\le\left(|\log\varepsilon|/\delta\right)^{1/\rho}}
		 n^{k(2+5\eta)}  \\
	& \lesssim\left( \delta^{-1} |\log\delta|^5 /
		|\log\varepsilon|^\chi \right)^k \left( 
		|\log\varepsilon|/\delta \right)^{(1+k(2+5\eta))/\rho} \\
	& = \delta^{ -(1+ k(2+5\eta+\rho))/\rho} |\log\delta|^{5k} 
		|\log\varepsilon|^{-(k(\rho\chi- 2-5\eta)-1)/\rho}. 
\end{split}
\end{equation}
Before we consider $R_2$, recall that we can choose $\rho\in(0\,,1)$ arbitrarily close to 1.
Since $\chi > 2$ (see the part (3) of Assumption \ref{ass:par}), 
we can choose $\eta \in (0\,,1)$ to be as small as we would like.
Hence, we may choose $\rho$ sufficiently close to 1 and $\eta$ sufficiently small
in order to ensure that
\begin{equation}\label{rho-chi}
  \rho \chi \;>\; 2 + 5\eta.
\end{equation} 
As regards $R_2$, we have 
\begin{equation}\label{R_2}
\begin{split}
	R_2 &\textstyle\le  \sum_{n>\left(|\log\varepsilon|/\delta\right)^{1/\rho}}
		\left(\delta^{-(1+\chi)} |\log\delta|^5 
		n^{-\rho\chi + 2+5\eta} \right)^k  \\
		 &\lesssim\left( \delta^{-(1+\chi)} |\log\delta|^5
		\right)^k \left( |\log\varepsilon|/\delta\right)^{(1-k(\rho\chi- 2-5\eta))/\rho},  \\
	& = \delta^{ -(1+ k(2+5\eta+\rho))/\rho}
		|\log\delta|^{5k} |\log\varepsilon|^{(1-k(\rho\chi- 2-5\eta))/\rho}. 
\end{split}
\end{equation}
In \eqref{R_1} and \eqref{R_2},  the implied constants do  not depend 
on the parameter $\varepsilon$ from the range given in \eqref{vareps}. By further insisting that  $\rho\chi > 2 + 5\eta$ --
see \eqref{rho-chi} -- and then choosing sufficiently large $k>6$, we can ensure that 
\[
 	k(\rho\chi- 2-5\eta) > 1.
\]
Because there exists $\chi>2$ such that
$\sup_{|z|\le1}  | \log z |^\chi\mathscr{E}(z)$ is finite
-- see Assumption \ref{ass:par} --  it follows fairly easily
that there exists $q>0$ such that
\begin{equation}\label{T1}
	\sum_{n=1}^\infty Q_1(n)
	\lesssim |\log\varepsilon|^{-q}.
\end{equation}
Indeed, because of the definition of $\delta$ -- see \eqref{delta} --
the quantity $\log(1/\delta)$ and $1/\delta$ does not contribute in a critical manner
and the bulk of the behavior of the preceding is governed by its dependence
on $\varepsilon$ through $\log(1/\varepsilon)$, to leading term. 
One could estimate $q$ explicitly of course, but this is not germane since
we can combine \eqref{T2}, \eqref{T4}, \eqref{T3}, and \eqref{T1} in order to find that
the worst error comes from \eqref{T2} since $\ell \ge 2$:
\[\textstyle
	\sum_{n=1}^\infty\P\left( \bm{A}_0\cap\bm{A}_1
	\cap\cdots\cap\bm{A}_n\cap\bm{A}_{n+1}^\mathsf{c}\right)
	\lesssim \left[\log\log\log(1/\varepsilon)\right]^{-1}.
\]
It follows readily from \eqref{P(w(T_1)} that
$\P(\bm{A}_0\cap\bm{A}_1^\mathsf{c})\lesssim
[\log\log\log(1/\varepsilon)]^{-1}$ as well, with room to spare. Therefore, we may conclude from
\eqref{aim} that
\[
	1-\lim_{N\to\infty}\mathcal{P}_N\lesssim 
	\left[ \log\log\log(1/\varepsilon)\right]^{-1},
\]
uniformly for all $\varepsilon\in(0\,,\exp(-\e^\e))$. Thanks to \eqref{A_n} and \eqref{P_N},
this implies that
\begin{equation*}
\begin{split}
	\P&\left(\left.  \|w(T_n)-u(T_n)\|_{C(\T)} / \|w(T_n)\|_{C(\T)}
	= O(\varepsilon_n)\text{ as $n\to\infty$} \ \right|\, \bm{A}_0\right) \\
   &\hspace{4cm} \ge 
	1 -  \text{\rm const} \left[ \log\log\log(1/\varepsilon)\right]^{-1},
\end{split}
\end{equation*}
where ``const'' does not depend on the parameter $\varepsilon$ that was fixed in \eqref{vareps}.
Let $\bm{B}$ denote the event whose conditional probability was just estimated. Clearly,
\[
	\P(\bm{B}\mid\bm{A}_0) \le 
	\P(\bm{B}) / \P(\bm{A}_0)
	\le (1-\varepsilon)^{-1} \P(\bm{B}) \le \P(\bm{B})/[1-\exp(-\e^\e)];
\]
see \eqref{vareps} and \eqref{P(A_0)}. It follows from this that
\[
	\P\left\{  \frac{\|w(T_n)-u(T_n)\|_{C(\T)}}{\|w(T_n)\|_{C(\T)}}
	= O(\varepsilon_n)\text{ as $n\to\infty$} \right\} \ge 
	1 - \frac{\text{\rm const}}{\log \log\log(1/\varepsilon)},
\]
where ``const'' does not depend on the parameter $\varepsilon$. Finally, we appeal
to Proposition \ref{pr:oscillation} in order to see that
for every $k\in[2\,,\infty)$,
\begin{align}\label{eq:sup:inf}\textstyle
	\sup_{R>0} \sup_{n\in\N} R^k\,\P\left\{ \sup_{x\in\T}w(T_n\,,x) \ge R \inf_{x\in\T} w(T_n\,,x)
	\right\} <\infty.
\end{align}
This and the Borel-Cantelli lemma together imply
that, for every $p>0$ there exists an a.s.-finite random variable $N_p$ such that
$\|w(T_n)\|_{C(\T)}\le n^p \inf_{x\in\T}w(T_n\,,x)$ for all $n\ge N_p$ almost surely,
and in particular,
\begin{equation}\label{ratio}
	\P\left\{ \left\| \frac{u(T_n)}{w(T_n)}-1\right\|_{C(\T)}
	= O\left(n^{p-2}\right)\text{ as $n\to\infty$} \right\}
	\ge\frac{\text{\rm const}}{\log \log\log(1/\varepsilon)},
\end{equation}
where ``const'' does not depend on $\varepsilon\in(0\,,\exp(-\e^\e))$;
see also \eqref{epsilon_n}. 

We now show that $\| (u(t)/w(t)) - 1\|_{C(\T)} \to 0$ as $t\to \infty$. Define, for $n\geq 0$, 
\begin{align*}
	K_n&=\|u(T_n)-w(T_n)\|_{C(\T)} ,  \\
	M_n&=K_n + (\varepsilon_n + \alpha_n\vee \varepsilon_n) \| w(T_n)\|_{C(\T)},\\ 
	\bm{B}_{n+1}&\textstyle=  \left\{ \omega\in\Omega:\ \sup_{T_{n}\leq t \leq T_{n+1}}
		\| u(t) - w(t) \|_{C(\T)}(\omega) \le M_n \right\}.
\end{align*}
Recall that $\alpha_n$ is defined in \eqref{alpha_n} (see also \eqref{Phi:W:n}) and $\varepsilon_n$ is defined in \eqref{epsilon_n}. 

We claim that for every $k>6$, there exists a constant $A=A(f\,, g)$ such that 
\begin{equation}\label{B_n}
	\P\big( \bm{B}_{n+1}^{\mathsf{c}} \cap \bm{A}_n \cap \bm{A}_0 \big)
	\leq A \left[ 1\wedge \varepsilon_n^{-1} \mathscr{E}\left(\e^{-\gamma T_n}\right)
	\right]^k 
	+ A\e^{-A^{-1}/ \sqrt{T_{n+1}-T_n}}. 
\end{equation}
Once we establish \eqref{B_n}, we proceed exactly as in the argument 
leading to \eqref{vareps}--\eqref{ratio}. Indeed, as can be seen from \eqref{T1},  there exists $q>0$ such that
\[\textstyle
	\sum_{n=1}^\infty \left[ \mathscr{E}\left(
	\e^{-\gamma T_n}\right)/\varepsilon_{n} \right]^k
	\leq \sum_{n=1}^\infty Q_1(n) \lesssim |\log \varepsilon|^{-q}. 
\]
On the other hand, \eqref{T3} shows that 
\[\textstyle
	\sum_{n=1}^\infty \exp(-A^{-1}/ \sqrt{T_{n+1}-T_n})
	=\sum_{n=1}^\infty Q_3(n) \lesssim (\log|\log \varepsilon|)^{-1/3}.
\]
Hence, with high probability,
\[
	\sup_{T_n \leq t \leq T_{n+1}} \frac{ \|w(t)-u(t)\|_{C(\T)} }{ \|w(T_n)\|_{C(\T)}}
	= O(\varepsilon_n\vee \alpha_n)\quad\text{as }n\to\infty,
\]
uniformly over all sufficiently large $n$.  At the same time, \eqref{eq:sup:inf} 
ensures that $\|w(T_n)\|_{C(\T)}/\inf_{x\in\T}w(T_n,x)$ is eventually bounded by 
a power $p$ of $n$ almost surely.  We can combine these facts to find that
for every $p\in(0\,,\eta)$,
\[
   \P\left\{ 
      \sup_{T_{n-1}\le t\le T_n} 
      \left\|\frac{u(t)}{w(t)} - 1\right\|_{C(\T)}
      \!\!\!\!
      = O\bigl(n^{p-\eta}\bigr)\;\text{as }n\to\infty
   \right\}
   \;\ge\;
   1 - \frac{\text{const}}{\bigl[\log\log\log(1/\varepsilon)\bigr]}.
\]
A Taylor expansion
of $\log$ at $x=1$ shows that \eqref{ratio} is equivalent to the statement of Theorem \ref{th:main},
and  completes the proof. Thus, it remains to prove \eqref{B_n}. 

In order to simply the typesetting, let $\lambda_n=\alpha_n \vee \varepsilon_n$
for all $n\in\Z_+$. 

First consider the case that $n=0$. Since $u\equiv v$ since they have
the same initial function, \eqref{B_n} follows from Proposition 
\ref{pr:tracking} (see also the estimation of $J_1$ below). 

Therefore, we need only consider $n\geq 1$,  in which case,
\[\textstyle
 	\P_{\mathscr{F}(T_n)}\left\{ \sup_{T_n \leq t \leq T_{n+1}}
	 \|u(t)-w(t)\|_{C(\T)}  \geq M_n  \right\} \leq J_1+ J_2,
\]
where 
\begin{align*}
	J_1&\textstyle= \P_{\mathscr{F}(T_n)}\left\{ \sup_{T_n \leq t \leq T_{n+1}}
		\|v(t) - w(t) \|_{C(\T)}  \geq \varepsilon_n\|w(T_n)\|_{C(\T)}   \right\},\\
	J_2&\textstyle=  \P_{\mathscr{F}(T_n)}\left\{ \sup_{T_n \leq t \leq T_{n+1}} \|u(t)-v(t)\|_{C(\T)}
		\geq K_n +   \lambda_n \| w(T_n)\|_{C(\T)}  \right\}.  
\end{align*}
Next, we estimate $J_1$ and $J_2$ separately and in turn.

\emph{Estimation of $J_1$.} By Proposition \ref{pr:tracking}, almost surely on $\bm{A}_0$
-- in particular also on $\{\tau(T_n)>T_n\}$ -- 
for every $k \in (6\,,\infty)$ there exists a constant $L=L(k\,,f\,,g)>0$ such that 
\begin{equation}\label{eq:J_1}
J_1 \leq  L[ \mathscr{E}(\e^{-\gamma T_n}) / \varepsilon_{n} ]^k.
\end{equation} 

\emph{Estimation of $J_2$.} 
Since $v(T_n)=w(T_n)$, we may apply the Markov property at time $T_n$
in order to be able to use Proposition \ref{pr:u-v} with $K=K_n=\|u(T_n)-w(T_n)\|_{C(\T)}$
and $\Lambda=\lambda_n \|w(T_n)\|_{C(\T)}$.  Since 
\[
K_n =\|u(T_n)-w(T_n)\|_{C(\T)} \leq \varepsilon_n \|w(T_n)\|_{C(\T)} \quad \text{on  $A_n$},
\]
there exists a constant $A = A(f, g)$ such that
\begin{equation}\label{eq:J_2-1}
J_2 \le A \exp \left(
  - \frac{\lambda_n^2}{ A (\varepsilon_n + \lambda_n) (\varepsilon_n + \lambda_n + \alpha_n) \sqrt{T_{n+1} - T_n}}\right).
\end{equation}
Because $\lambda_n = \alpha_n \vee \varepsilon_n$, we have 
\begin{equation}\label{eq:lambda-square}
\lambda_n^2 \ge 
\frac{1}{6}(\varepsilon_n + \lambda_n)(\varepsilon_n + \lambda_n + \alpha_n).
\end{equation}
Combining \eqref{eq:J_2-1} and \eqref{eq:lambda-square} gives
\begin{equation}\label{eq:J_2-2}
J_2 \le 
A \exp\left(-\frac{1}{6A\sqrt{T_{n+1}-T_n}}\right).
\end{equation}
Inequality \eqref{B_n} follows readily from the above estimates, 
\eqref{eq:J_1} for $J_1$  and \eqref{eq:J_2-2} for $J_2$ .
This completes our proof.\qed 
\appendix
\section{}
Let us conclude with the brief statement and proof of the following elementary inequality.

\begin{lemma}\label{lem:E}
	$\exists c>0\,  \forall\delta\in(0\,,1): \sum_{n=3}^\infty\exp\left( - \delta^{-1/2} (\log n)^{3/2} \right)\le
	c\delta^{1/3}.$
\end{lemma}

\begin{proof}
	By the integral test of calculus, the sum above is 
	\begin{align*}
		&\textstyle\le\int_\e^\infty\exp\left( - c\delta^{-1/2} (\log y)^{3/2} \right)\d y
			< \int_0^\infty\exp\left(x -\{x^{3/2}/\sqrt\delta\}\right)\d x\\
		&\textstyle=\delta^{1/3}\int_0^\infty\exp\left(\delta^{1/3} z - z^{3/2}\right)\d z
			\le c\delta^{1/3}
			\qquad\forall \delta\in(0\,,1),
	\end{align*}
	where $c=\int_0^\infty\exp(z-z^{3/2})\,\d z$.
\end{proof}


\begin{thebibliography}{10}

\bibitem{GK2}
\'Eric Brunet, Yu~Gu, and Tomasz Komorowski, \emph{High temperature behaviors
  of the directed polymer on a cylinder}, ArXiv (2021), arxiv:07368.

\bibitem{CM94}
Ren\'e{}~A. Carmona and S.~A. Molchanov, \emph{{P}arabolic {A}nderson {P}roblem
  and {I}ntermittency}, Mem. Amer. Math. Soc. \textbf{108} (1994), no.~518,
  viii+125. \MR{1185878}

\bibitem{ChenDalang2014}
Le~Chen and Robert~C. Dalang, \emph{H\"older-continuity for the nonlinear
  stochastic heat equation with rough initial conditions}, Stoch. Partial
  Differ. Equ. Anal. Comput. \textbf{2} (2014), no.~3, 316--352. \MR{3255231}

\bibitem{CJKS2014}
Daniel Conus, Mathew Joseph, Davar Khoshnevisan, and Shang-Yuan Shiu,
  \emph{Initial measures for the stochastic heat equation}, Ann. Inst. Henri
  Poincar\'e{} Probab. Stat. \textbf{50} (2014), no.~1, 136--153. \MR{3161526}

\bibitem{DauvergneOrtmannVirag}
Duncan Dauvergne, Janosch Ortmann, and B\'alint Vir\'ag, \emph{The directed
  landscape}, Acta Math. \textbf{229} (2022), no.~2, 201--285. \MR{4554223}

\bibitem{Fife}
Paul~C. Fife, \emph{Mathematical aspects of reacting and diffusing systems},
  Lecture Notes in Biomathematics, vol.~28, Springer-Verlag, Berlin-New York,
  1979. \MR{527914}

\bibitem{FoondunKhoshnevisan2009}
Mohammud Foondun and Davar Khoshnevisan, \emph{Intermittence and nonlinear
  parabolic stochastic partial differential equations}, Electron. J. Probab.
  \textbf{14} (2009), no. 21, 548--568. \MR{2480553}

\bibitem{Garsia}
Adriano~M. Garsia, \emph{Continuity properties of {G}aussian processes with
  multidimensional time parameter}, Proceedings of the {S}ixth {B}erkeley
  {S}ymposium on {M}athematical {S}tatistics and {P}robability ({U}niv.
  {C}alifornia, {B}erkeley, {C}alif., 1970/1971), {V}ol. {II}: {P}robability
  theory, Univ. California Press, Berkeley, CA, 1972, pp.~369--374. \MR{410880}

\bibitem{GK2022}
Yu~Gu and Tomasz Komorowski, \emph{High temperature behaviors of the directed
  polymer on a cylinder}, J. Stat. Phys. \textbf{186} (2022), no.~3, Paper No.
  48, 15. \MR{4379456}

\bibitem{GK}
Yu~Gu and Tomasz Komorowski, \emph{{KPZ} on torus: Gaussian fluctuations},
  ArXiv (2023), arxiv:13540.

\bibitem{GK2024}
\bysame, \emph{{S}ome recent progress on the periodic {KPZ} equation}, ArXiv
  (2024), arXiv:2408.14174.

\bibitem{Hairer2013}
Martin Hairer, \emph{Solving the {KPZ} equation}, Ann. of Math. (2)
  \textbf{178} (2013), no.~2, 559--664. \MR{3071506}

\bibitem{KKM2023:Oscillations}
Davar Khoshnevisan, Kunwoo Kim, and Carl Mueller, \emph{Dissipation in
  parabolic {SPDE}s {II}: {O}scillation and decay of the solution}, Ann. Inst.
  Henri Poincar\'e{} Probab. Stat. \textbf{59} (2023), no.~3, 1610--1641.
  \MR{4635721}

\bibitem{KKM2024}
\bysame, \emph{On the valleys of the stochastic heat equation}, Ann. Appl.
  Probab. \textbf{34} (2024), no.~1B, 1177--1198. \MR{4700256}

\bibitem{KKMS}
Davar Khoshnevisan, Kunwoo Kim, Carl Mueller, and Shang-Yuan Shiu,
  \emph{Dissipation in parabolic {SPDE}s}, J. Stat. Phys. \textbf{179} (2020),
  no.~2, 502--534. \MR{4091567}

\bibitem{KKM2023a}
\bysame, \emph{Phase analysis for a family of stochastic reaction-diffusion
  equations}, Electron. J. Probab. \textbf{28} (2023), Paper No. 101, 66.
  \MR{4620551}

\bibitem{Liggett}
Thomas~M. Liggett, \emph{Interacting {P}article {S}ystems}, Classics in
  Mathematics, Springer-Verlag, Berlin, 2005, Reprint of the 1985 original.
  \MR{2108619}

\bibitem{Mueller1}
Carl Mueller, \emph{On the support of solutions to the heat equation with
  noise}, Stochastics Stochastics Rep. \textbf{37} (1991), no.~4, 225--245.
  \MR{1149348}

\bibitem{Mueller2}
\bysame, \emph{Coupling and invariant measures for the heat equation with
  noise}, Ann. Probab. \textbf{21} (1993), no.~4, 2189--2199. \MR{1245306}

\bibitem{MMP}
Carl Mueller, Leonid Mytnik, and Edwin Perkins, \emph{Nonuniqueness for a
  parabolic {SPDE} with {$\frac{3}{4}-\varepsilon$}-{H}\"older diffusion
  coefficients}, Ann. Probab. \textbf{42} (2014), no.~5, 2032--2112.
  \MR{3262498}

\bibitem{MN}
Carl Mueller and David Nualart, \emph{Regularity of the density for the
  stochastic heat equation}, Electron. J. Probab. \textbf{13} (2008), no. 74,
  2248--2258. \MR{2469610}

\bibitem{BlessingRosati}
Alexandra~Blessing (Neam\c{t}u) and Tommaso Rosati, \emph{{Q}uantitative
  instability for stochastic scalar reaction-diffusion equations}, ArXiv
  (2024), arXiv:2406.04651.

\bibitem{QuastelSarkar}
Jeremy Quastel and Sourav Sarkar, \emph{Convergence of exclusion processes and
  the {KPZ} equation to the {KPZ} fixed point}, J. Amer. Math. Soc. \textbf{36}
  (2023), no.~1, 251--289. \MR{4495842}

\bibitem{RevuzYor}
Daniel Revuz and Marc Yor, \emph{{C}ontinuous {M}artingales and {B}rownian
  {M}otion}, third ed., Grundlehren der mathematischen Wissenschaften
  [Fundamental Principles of Mathematical Sciences], vol. 293, Springer-Verlag,
  Berlin, 1999. \MR{1725357}

\bibitem{Salins}
Michael Salins, \emph{Solutions to the stochastic heat equation with
  polynomially growing multiplicative noise do not explode in the critical
  regime}, Ann. Probab. \textbf{53} (2025), no.~1, 223--238. \MR{4852006}

\bibitem{Shiga1994}
Tokuzo Shiga, \emph{Two contrasting properties of solutions for one-dimensional
  stochastic partial differential equations}, Canad. J. Math. \textbf{46}
  (1994), no.~2, 415--437. \MR{1271224}

\bibitem{Walsh}
John~B. Walsh, \emph{{A}n {I}ntroduction to {S}tochastic {P}artial
  {D}ifferential {E}quations}, \'{E}cole d'\'{e}t\'{e} de probabilit\'{e}s de
  {S}aint-{F}lour, {XIV}---1984, Lecture Notes in Math., vol. 1180, Springer,
  Berlin, 1986, pp.~265--439. \MR{876085}

\bibitem{Young12}
W.~H. Young, \emph{On the multiplication of successions of {F}ourier
  constants}, Proc. Roy. Soc. London, Ser. A \textbf{87} (1912), 331--339.

\bibitem{ZimmermannEtAl}
Martin~G. Zimmermann, Ra\'ul Toral, Oreste Piro, and Maxi San~Miguel,
  \emph{Stochastic spatiotemporal intermittency and noise-induced transition to
  an absorbing phase}, Phys. Rev. Lett. \textbf{85} (2000), no.~17, 3621--3615.

\end{thebibliography}

\providecommand{\bysame}{\leavevmode\hbox to3em{\hrulefill}\thinspace}
\providecommand{\MR}{\relax\ifhmode\unskip\space\fi MR }
\providecommand{\MRhref}[2]{%
  \href{http://www.ams.org/mathscinet-getitem?mr=#1}{#2}
}
\providecommand{\href}[2]{#2}

\end{document}